\documentclass[12pt]{article}
\usepackage[margin=1in]{geometry}
\usepackage{amstext,amsthm,amsmath,amsfonts,amssymb,amscd,latexsym,txfonts,stmaryrd}
\usepackage{color}
\usepackage{epsfig}
\usepackage{cite}
\usepackage[all,cmtip]{xy}
\usepackage{appendix}
\usepackage{enumerate}
\usepackage{ytableau}
\usepackage[colorlinks=true, pdfstartview=FitV, linkcolor=blue, citecolor=blue, urlcolor=blue]{hyperref}

\newcommand{\flo}[1]{\left\lfloor\frac{#1}{2}\right\rfloor}
\newcommand{\supp}{\operatorname{supp}}
\newcommand{\ydeal}{(y_1,\ldots,y_{n-1})^{(k)}+(x_{n}^2)}
\renewcommand{\d}{\mathbf{d}}

\newcommand{\y}{\mathbf{y}}
\newcommand{\m}{\mathbf{m}}
\renewcommand{\b}{\mathbf{b}}
\newcommand{\x}{\mathbf{x}}
\renewcommand{\a}{\mathbf{a}}
\newcommand{\stab}{\operatorname{STab}}
\newcommand{\tab}{\operatorname{Tab}}

\newcommand{\dsp}{\displaystyle}

\newcommand{\C}{{\mathbb C}}

\newcommand{\Z}{{\mathbb Z}}

\newcommand{\N}{{\mathbb N}}

\newcommand{\F}{{\mathbb F}}

\theoremstyle{theorem}
\newtheorem*{theoremA*}{Theorem A}
\newtheorem*{theoremB*}{Theorem B}
\newtheorem*{theoremB1*}{Theorem B1}
\newtheorem*{theoremC*}{Theorem C}
\newtheorem*{theoremC1*}{Theorem C1}
\newtheorem*{theoremD*}{Theorem D}
\newtheorem*{theorem*}{Theorem}
\newtheorem*{proposition*}{Proposition}
\newtheorem{theorem}{Theorem}[section]
\newtheorem{proposition}[theorem]{Proposition}
\newtheorem{lemma}[theorem]{Lemma}
\newtheorem{corollary}[theorem]{Corollary}

\newtheorem*{conjecture*}{Conjecture}
\newtheorem{conjecture}[theorem]{Conjecture}
\newtheorem*{fact*}{Fact}
\newtheorem{fact}[theorem]{Fact}

\theoremstyle{definition}
\newtheorem{definition}[theorem]{Definition}

\newtheorem{example}[theorem]{Example}

\newtheorem*{remark*}{Remark}
\newtheorem*{remarks*}{Remarks}

\begin{document}

\title{Principal Radical Systems, Lefschetz Properties, and Perfection of Specht Ideals of Two-Rowed Partitions}
\author{Chris McDaniel and Junzo Watanabe}
%\address{Dept. of Math. and Stat.\\
%University of Massachusetts\\
%Amherst, MA 01003}
%\email{mcdaniel@math.umass.edu}
\date{}
	\maketitle

\abstract{We show that the Specht ideal of a two-rowed partition is perfect over an arbitrary field, provided that the characteristic is either zero or bounded below by the size of the second row of the partition, and we show this lower bound is tight.  We also establish perfection and other properties of certain variants of Specht ideals, and find a surprising connection to the weak Lefschetz property. Our results in particular give a self-contained proof of Cohen-Macaulayness of certain $h$-equals sets, a result previously obtained by Etingof-Gorsky-Losev over the complex numbers using rational Cherednik algebras. }
\vspace{.2cm}


\section{Introduction}
Fix an integer $n$, let $\F$ be any field, and let $R=\F[x_1,\ldots,x_n]$ be the polynomial ring with its standard grading, and equipped with the usual action of the symmetric group $\mathfrak{S}_n$ by permuting the variables.  For any partition $\lambda\vdash n$ the Specht module $V(\lambda)$ over $\F$ is the $\F$-vector space generated by the Specht polynomials of $\lambda$ which are indexed by the set of tableaux $T$ on the Young diagram of $\lambda$.  If $\F$ has characteristic zero, then the Specht modules form a complete list of irreducible $\mathfrak{S}_n$-representations, highlighting their importance in representation theory.  In this paper we take the point of view of commutative algebra, and study the ideals generated by Specht modules called Specht ideals.

Specifically we show that for partitions with two parts $\lambda=(n-k,k)$ (or Young diagrams with two rows), the Specht ideal, which we denote by $\mathfrak{a}(n,k,k)$ is radical and, if the characteristic of $\F$ is zero or sufficiently large, perfect.  Our results are stated in terms of commutative algebra, but they can be interpreted geometrically as follows:
\begin{proposition}
	\label{prop:EGLgen}
	Fix a field $\F$ with $\operatorname{char}(\F)=p\geq 0$, and fix $m\geq 3$.  For each $2\leq h\leq m$, define $X_{m,h}\subset\F^{m}$ as the union of $\mathfrak{S}_{m}$ translates of the linear subspace cut out by the $h-1$ linear equations $x_1=\cdots =x_h$, i.e.
	$$X_{m,h}=\bigcup_{\sigma\in\mathfrak{S}_{m}}\sigma.\left\{(x_1,\ldots,x_m)\in\F^m \ | \ x_1=\cdots =x_h\right\}.$$
	Assume that $h>\flo{m+1}$.  
	\begin{enumerate}
		\item If $p=0$ or $p\geq m-h+1$, then $X_{m,h}$ is Cohen-Macaulay\footnote{meaning that its homogeneous coordinate ring is Cohen-Macaulay}.
		
		\item If $0<p<\flo{m}$, then $X_{m,m-p}$ is not Cohen-Macaulay.
	\end{enumerate}
\end{proposition}
Over the field $\F=\C$, Proposition \ref{prop:EGLgen}(1.) was obtained by Etingof-Gorsky-Losev \cite[Proposition 3.11]{EGL}, using deep results from the representation theory of rational Cherednik algebras.  Our proof of Proposition \ref{prop:EGLgen} is more elementary in that it uses only basic results from commutative algebra, which we hope will appeal to those uninitiated with rational Cherednik algebras.  We emphasize however that our elementary proof is by no means easy.

%Our assumption that $h>\flo{m+1}$ is the equivalent to the condition that the associated Specht ideals correspond to two-rowed partitions.  Furthermore, Yanagawa has shown \cite{Yana} that the implication \emph{$X_{m,h}$ is Cohen-Macaulay implies $h>\flo{m}$} is independent of the field.  Proposition \ref{prop:EGLgen} then allows one to extend Theorem \eqref{EGL} to arbitrary fields $\F$ in the two-rowed case, and takes a step toward classification of the field characteristics $p\geq 0$ for which $X_{m,h}$ is perfect.  Interestingly, the proof of Theorem \ref{EGL} given by Etingof-Gorsky-Losev uses deep results from the representation theory of rational Cherednik algebras.  Our proof of Proposition \ref{prop:EGLgen} is elementary in the sense that it uses only basic facts from commutative algebra, and does not rely on Cherednik algebras.
%Besides applying to fields of positive characteristic, our proof has a more elementary flavor, using only basic facts from commutative algebra and circumventing the need for rational Cherednik algebras.  
%We emphasize however that our elementary proof is by no means easy.

The study of Specht ideals seems to have been initiated by Yanagawa in his recent paper \cite{Yana}, although they have appeared implicitly in the earlier works of others \cite{BGS,EGL,FS}.  As Yanagawa shows in his paper, the connection between Proposition \ref{prop:EGLgen} and Specht ideals is as follows:  Setting $m=n+1$ and $k+1=m-h+1$, and if $I_{m,h}\subset R$ is the ideal cutting out the union of linear spaces $X_{m,h}$ then we have   
\begin{equation}
\label{eq:Inh}
I_{m,h}=\bigcap_{\sigma\in \mathfrak{S}_m}\sigma.\left(x_1-x_2,\ldots,x_1-x_h\right)=\sqrt{\mathfrak{a}(n+1,k+1,k+1)}.
\end{equation}
In his paper \cite{Yana}, Yanagawa proves that two-rowed Specht ideals are radical by an ingenious but complicated argument.  He then invokes the Etingof-Gorsky-Losev result \cite[Proposition 3.11]{EGL} to prove that the Specht ideal $\mathfrak{a}(n+1,k+1,k+1)$ (although he used different notation) is perfect if the field has characteristic zero.  The present paper grew out of an attempt to understand and simplify Yanagawa's arguments and to find an elementary proof of the Etingof-Gorsky-Losev result in the two-rowed case.  As the reader will surmise, the distinguishing feature of Specht ideals of two-rowed partitions is that their minimal generators are square-free, a fact which will be exploited throughout this paper.   

%Yanagawa also shows in that paper \cite{Yana} that certain Specht ideals, including those corresponding to two-rowed partitions, are always radical.  
%In conjunction with , and we will prove a generalization of that result here.  
Recall that an ideal $I\subset R$ in a Noetherian ring is called \emph{perfect} if its grade is equal to its homological (or projective) dimension.  In a polynomial ring $R$, a homogeneous ideal $I\subset R$ is perfect if and only if its quotient $R/I$ is Cohen-Macaulay. The following is one of the main results of this paper, and is the algebraic analogue of Proposition \ref{prop:EGLgen}.
\begin{theorem}
	\label{thm:SpechtP}
	Let $\F$ be any field of characteristic $p\geq 0$, and fix positive integers $n,k$ satisfying $n\geq 2k+1$.
	\begin{enumerate}
		\item If $p=0$ or $p\geq k+1$, then the Specht ideal $\mathfrak{a}(n+1,k+1,k+1)$ is perfect.
		\item If $n\geq 2p+1$, then the Specht ideal $\mathfrak{a}(n+1,p+1,p+1)$ is not perfect. 
	\end{enumerate} 
\end{theorem}
In his paper \cite{Yana} Yanagawa has conjectured that the Specht ideal $\mathfrak{a}(n+1,k+1,k+1)$ is perfect in characteristic $p$ \emph{if and only if} $p=0$ or $p\geq k+1$.  Theorem \ref{thm:SpechtP} proves one implication and part of the other one in Yanagawa's conjecture.
  
Our proof of Theorem \ref{thm:SpechtP} is inspired by the seminal paper of Hochster-Eagon \cite{HE}, in which they proved perfection of generic determinantal ideals using what they termed a principal radical system.  Our method, which might be more aptly described as a \emph{deconstructed} principal radical system is based on the following elementary facts from commutative algebra:
\begin{lemma}
	\label{lem:HE1}
	Let $I\subset R$ be a homogeneous ideal, and let $x\in R\setminus I$ be a homogeneous polynomial.
	\begin{enumerate}
		\item If $(I:x)=I$ and $I+(x)$ is radical, then $I$ is also radical.
		\item If $(I:x)= I$ and $I+(x)$ is perfect, then $I$ is also perfect.
	\end{enumerate}
\end{lemma}

\begin{lemma}
	\label{lem:HE2}
	Suppose that ideals $I,J\subset R$ are homogeneous ideals, both perfect of the same grade $g$ and suppose that $I+J$ has grade $g+1$. 
	Then $I \cap J$ is perfect if and only if $I+J$ is perfect. 
\end{lemma}
Actually we use a slight modification of Lemma \ref{lem:HE1}, but we wrote it that way to emphasize the connection between radical and perfection.  In fact, before we prove that our Specht ideals are perfect, we must first prove that they are radical.  More specifically, the main challenge in applying Lemma \ref{lem:HE1} is finding the right principal ideal $(x)$ with respect to $I$, and it is this issue that Lemma \ref{lem:HE2} addresses.  However, in applying Lemma \ref{lem:HE2}, we must know what amounts to a primary decomposition of our ideal $I$.  Our method then consists of repeated application of Lemma \ref{lem:HE1} and Lemma \ref{lem:HE2} to obtain a sequence of (families of) ideals beginning with the Specht ideal, and ending with some ideal for which perfection is already known, e.g. by induction.  A schematic of our argument with our sequence of ideals is shown below:  
$$\xymatrixcolsep{3pc}\xymatrix{\mathfrak{a}(n+1,k+1,k+1)\ar[r]^-{\text{Lem. \ref{lem:HE1}}} & \mathfrak{a}(n,k,k+1) \ar[r]^-{\text{Lem. \ref{lem:HE2}}} & I(n,k) \ar[r]^-{\text{Lem. \ref{lem:HE2}}} & J(n,k)\ar[dl]^-{\text{Lem. \ref{lem:HE1}}}\\ && I(n-1,k-1)}$$
We give a more detailed description of our method and these ideals listed above, together with the other results of this paper, below.
%To this end, we introduce for $0\leq k\leq d\leq n-k$ the \emph{$d$-shifted Specht ideal} $\mathfrak{a}(n,k,d)$ generated by square-free products of Specht polynomials of type $\lambda=(n-k,k)$ and square-free monomials of degree $d-k$.  These shifted Specht ideals interpolate between Specht ideals, in case $d=k$, and square-free monomial ideals, in case $k=0$, where $\mathfrak{a}(n,0,d)=(x_1,\ldots,x_n)^{(d)}$ the ideal generated by all square-free monomials of fixed degree $d$.  Just as minimal generators for the Specht ideal are indexed by standard Young tableaux  on $\lambda$, we show that minimal generators for the shifted Specht ideal are indexed by standard Young tableaux on a shifted version of $\lambda$.   
%In his paper \cite{Yana}, Yanagawa proves that the ideal $\mathfrak{a}(n,k,k+1)$, which we call a \emph{shifted Specht ideal}, satisfies the following decomposition formula  
%He further proves, by an ingenious but complicated argument, that the following decomposition formula holds: 
%\begin{equation}
%\label{eq:ydecomp}
%\mathfrak{a}(n,k,k+1)=\mathfrak{a}(n,k,k)\cap (x_1,\ldots,x_n)^{(k+1)}
%\end{equation}
%but decomposition \eqref{eq:ydecomp} holds if and only if $\mathfrak{a}(n,k,k+1)$, is radical.  

Taking $I=\mathfrak{a}(n+1,k+1,k+1)$ and $x=x_{n+1}$ in Lemma \ref{lem:HE1}, it is easy to show that $(I:x)=I$, and that $I+(x)=\mathfrak{a}(n,k,k+1)+(x_{n+1})$ where $\mathfrak{a}(n,k,k+1)$ is the ideal generated by square-free products of Specht polynomials of type $\lambda=(n-k,k)$ with a linear monomial in the variables $x_1,\ldots,x_n$.  Generalizing, we introduce, for integers $0\leq k\leq d\leq n-k$, the \emph{$d$-shifted Specht ideal} $\mathfrak{a}(n,k,d)$ generated by square-free products of Specht polynomials of type $\lambda=(n-k,k)$ and square-free monomials of degree $d-k$.  

These shifted Specht ideals interpolate between Specht ideals $\mathfrak{a}(n,k,k)$ in case $d=k$, and square-free monomial ideals in case $k=0$, where $\mathfrak{a}(n,0,d)=(x_1,\ldots,x_n)^{(d)}$ is the ideal generated by all square-free monomials of fixed degree $d$.  Our first step in understanding these shifted Specht ideals is to find a minimal generating set.  Just as minimal generators for the Specht ideal are indexed by standard Young tableaux on $\lambda$, we show that minimal generators for the shifted Specht ideal are indexed by standard Young tableaux on a shifted version of $\lambda$.   
%The shifted Specht ideals interpolate between Specht ideals $\mathfrak{a}(n,k,k)$ if $d=k$ and if $k=0$, the square-free monomial ideal $\mathfrak{a}(n,0,d)=(x_1,\ldots,x_n)^{(d)}$ generated by square-free monomials of fixed degree $d$.  Our first step in understanding these shifted Specht ideals is to find a minimal generating set.  It is well known that a minimal generating set for the Specht ideal $\mathfrak{a}(n,k,,k)$ is indexed by standard tableaux on the Young diagram for $\lambda=(n-k,k)$.  We show that the shifted Specht ideal $\mathfrak{a}(n,k,d)$ has a minimal generating set indexed by standard tableaux on the Young diagram for a shifted version of $\lambda$.
\begin{theorem}
	\label{thm:A}
	A minimal generating set for the shifted Specht ideal $\mathfrak{a}(n,k,d)$ is formed by the shifted Specht polynomials $F_T(d)$ indexed by standard tableaux $T$ on the $d$-shifted shape of $\lambda=(n-k,k)$ obtained from Young diagram of $\lambda$ by moving the last $n-k-d$ boxes on the top row to the first $n-k-d$ boxes on the bottom row, i.e.
	$$\lambda(d)=\underbrace{\ydiagram{0,3}}_{n-k-d}\overbrace{\ydiagram{6,3}}^d.$$
\end{theorem}
The linear span of the shifted Specht polynomials $F_T(d)$, $T\in\operatorname{Tab}(\lambda(d))$ forms an $\mathfrak{S}_n$-representation $V(n,k,d)$ that we call a \emph{shifted Specht module}, which is reducible in general.  We prove that our $d$-shifted Specht ideals satisfy the following decomposition formula, which is crucial in our quest for perfection, and which holds if and only if our shifted Specht ideals are radical.
%We use Theorem \ref{thm:A} to extend Yanagawa's decomposition formula \eqref{eq:ydecomp} to all shifted Specht ideals:
\begin{theorem}
	\label{thm:B1}
	For any integers $k,d$ satisfying $1\leq k< d\leq n-k$, we have 
	\begin{equation}
	\label{eq:nkd}
	\mathfrak{a}(n,k,d)=\mathfrak{a}(n,k,d-1)\cap (x_1,\ldots,x_n)^{(d)}.
	\end{equation}
\end{theorem}
Yanagawa \cite{Yana} has proved Theorem \ref{thm:B1} in the special case $d=k+1$ using a clever argument, which is described in further detail below.  As it turns out, his argument goes through verbatim to prove Theorem \ref{thm:B1} in the general case, and is in fact simplified by Theorem \ref{thm:A}.  It follows directly from Theorem \ref{thm:B1} that our shifted Specht ideals are radical.
\begin{theorem}
	\label{thm:B}
	Fix integers $k,d$ satisfying $0\leq k\leq d\leq n-k$.  Then the shifted Specht ideal $\mathfrak{a}(n,k,d)$ is radical. 
\end{theorem}

%Theorem \ref{thm:B1} was proved by Yanagawa \cite{Yana} in the special case $d=k+1$, using an ingenious argument that we describe in general terms below.  To prove that an ideal $I$ satisfies a decomposition formula of the form $I=I'\cap J$ where $J$ is a monomial ideal, first show that the intersection $I'\cap J$ can be generated by products of minimal generators of $I'$ with monomials (not necessarily from $J$), with the additional property that if a sum of such products is in  $I'\cap J$ then each of the summands is also in $I'\cap J$ (perhaps we should call $I'\cap J$ an $I'$ monomial ideal).  Next fix a monomial $m$ and split the minimal generators $V(I')$ into two parts say $V(I')=V_m(I')\oplus V^m(I')$, determined by whether or not the minimal generators contain $m$ in their support.  In our case we can show that if $\nu\in V_m$ then $m\cdot \nu\in I$, and, with more effort, that if $\nu\in V^m$ and $m\cdot\nu\in I'\cap J$ then $m\cdot \nu\in I$.  We highlight this argument here because it applies not only to the proof of Theorem \ref{thm:B1} but also to the proof of Theorem \ref{thm:C}(3.) below, and hence it seems to be important in the theory of two-rowed Specht ideals. 

%Using Theorem \ref{thm:B1} in conjunction with Lemma \ref{lem:HE1}(1.), one can easily prove that shifted Specht ideals, including all two-rowed Specht ideals are radical:

Perfection of shifted Specht ideals is more difficult to prove, and in fact, most shifted Specht ideals are not perfect.  Indeed Theorem \ref{thm:B1} implies that the shifted Specht ideal $\mathfrak{a}(n,k,d)$ does not have pure height and hence cannot be perfect if $d\neq k,k+1$.  To show that the shifted Specht ideal $\mathfrak{a}(n,k,k+1)=\mathfrak{a}(n,k,k)\cap (x_1,\ldots,x_n)^{(k+1)}$ is perfect, we appeal to Lemma \ref{lem:HE2} and introduce the \emph{Specht-monomial ideal} $I(n,k)=\mathfrak{a}(n,k,k)+(x_1,\ldots,x_n)^{(k+1)}$.  

While the Specht-monomial ideal is not radical in general, it does satisfy a decomposition formula similar to \eqref{eq:nkd}, but it depends on the field characteristic.  We were pleased to discover that this dependence on field characteristic is the same one imposed by the weak Lefschetz property of certain monomial complete intersection algebras.  This is summarized in the following, which can be considered the other main result of this paper.
%, and in fact a modified version of Yanagawa's argument outlined above can be used to prove it.  There is one small caveat however:  the decomposition formula for the Specht-monomial ideal depends on the characteristic of the field $\F$.   The surprise here is that this dependence on characteristic is characterized by the weak Lefschetz property of a certain algebra.\footnote{The authors of this paper met at a workshop on Lefschetz properties, and they were especially pleased to find this connection.}  From this decomposition we prove the perfection of the Specht-monomial ideal, and conversely.  This is summarized in the following theorem, which is the other main result of this paper.
\begin{theorem}
	\label{thm:C}
	Let $\F$ be a field with $\operatorname{char}(\F)=p\geq 0$, and let $n$ and $k$ be positive integers satisfying $n\geq 2k+1$.  The following are equivalent:
	\begin{enumerate}
		\item $p=0$ or $p\geq k+1$.
		
		\item The quadratic monomial complete intersection 
		\begin{equation}
		\label{eq:mci}
		C=\frac{\F[x_1,\ldots,x_{2k}]}{(x_1^2,\ldots,x_{2k}^2)}
		\end{equation}
		has the weak Lefschetz property.
		
		\item The Specht-monomial ideal $I(n,k)$ satisfies decomposition 
		\begin{align}
		\label{eq:pdecomp}
		I(n,k)= & I(n-1,k-1)\cap \left(\ydeal\right),\\	
		\nonumber\text{where} \ \ \ & y_i=\begin{cases} x_n-x_i & \text{if} \ 1\leq i\leq n-1\\ x_n & \text{if} \ i=n\\ \end{cases}
		\end{align}
		
		\item The Specht-monomial ideal $I(n,k)$ is perfect.
	\end{enumerate}	
\end{theorem}
We shall break Theorem \ref{thm:C} into three parts, one for each equivalence (1.) $\Leftrightarrow$ (a.) which shall refer to hereafter as Theorem \ref{thm:C}(a.) for a=2,3,4.  Theorem \ref{thm:C}(2.) is due to Kustin-Vraciu \cite{KV}, and we will not re-prove it here.  Theorem \ref{thm:C}(3.) and Lemma \ref{lem:HE2} lead to yet another family of ideals, which remain unnamed:
$$J(n,k)=I(n-1,k-1)+\ydeal.$$
Finally to prove Theorem \ref{thm:C}(4.), and hence also Theorem \ref{thm:SpechtP}(1.) we apply Lemma \ref{lem:HE1} to show that our unnamed ideal $J(n,k)$ is perfect.  Unfortunately, Theorem \ref{thm:SpechtP}(2.) seems to be the best we could do with our method, although we conjecture that the full converse of Theorem \ref{thm:SpechtP}(1.) holds. 
\begin{conjecture}
	\label{conj:WM}
	If $p=\operatorname{char}(\F)$ and $n,k$ are positive integers satisfying $0<p<k+1$ and $n\geq 2k+1$, then the Specht ideal $\mathfrak{a}(n+1,k+1,k+1)$ is not perfect. 
\end{conjecture}
\noindent Theorem \ref{thm:SpechtP}(2.) says Conjecture \ref{conj:WM} holds for $k=p$, but it remains open for $k>p$.  Together with Theorem \ref{thm:SpechtP}, Conjecture \ref{conj:WM} is equivalent to Yanagawa's conjecture \cite[Conjecture 5.5]{Yana}.

Some further remarks on the connection to the weak Lefschetz property are in order here.  Since (shifted) Specht polynomials are square-free they are identified with elements of the algebra $A=\F[x_1,\ldots,x_n]/(x_1^2,\ldots,x_n^2)$, which carries an $\mathfrak{sl}_2$-representation in which the raising operator is multiplication by the sum of variables, and the lowering operator is the corresponding linear partial differential operator.  Moreover surjectivity of this lowering operator in degree $k$ is equivalent to equality of the kernel of that lowering operator with the Specht module $V(n,k,k)$, which is in turn equivalent to the weak Lefschetz property of $C$ in \eqref{eq:mci}.  Surjectivity of the lowering operator on $A$ is key to proving decomposition \eqref{eq:pdecomp}, and in fact reveals a hidden property of the Specht-monomial ideal:  in small positive characteristic, i.e. $0<p<k+1$ the ideal $I(n,k)$ has an embedded prime divisor, a phenomenon which does not occur for the Specht ideals.     

%We make some additional remarks on the connection to Lefschetz properties here.  One of the distinguishing features of Specht ideals of two-rowed partitions is that their minimal generating sets, i.e. the Specht polynomials, are square-free.  As it happens, the set of square-free monomials form a basis for the Artinian algebra $A=\F[x_1,\ldots,x_n]/(x_1^2,\ldots,x_n^2)$, and hence we can identity our Specht polynomials with elements in $A$.  Moreover, $A$ carries an $\mathfrak{sl}_2$-representation in which the raising operator is multiplication by the sum of variables and the lowering operator is the associated linear differential operator.  One can then show that the Specht module $V(n,k,k)$ is equal to the kernel of the lowering operator if and only if the algebra $C$ in \eqref{eq:mci} has the weak Lefschetz property.  This lowering operator plays a key role in our proof of Theorem \ref{thm:C}, and also reveals a hidden property of Specht-monomial ideals: in relatively small characteristic, i.e. $0<p<k+1$, the ideal $I(n,k)$ has an embedded prime divisor.  This is one explanation why imperfection for Specht-monomial ideals is easier to detect than for Specht ideals:  Specht ideals can \emph{never} have embedded prime divisors because they are radical in all characteristics.

The most technically difficult parts of our arguments are the decompositions in Theorem \ref{thm:B1} and Theorem \ref{thm:C}(3.).  The two proofs we give are strikingly similar, and are both based on that ingenious and clever argument of Yanagawa mentioned above.  This argument, in general terms, runs as follows:  
To prove that an ideal $I$ satisfies a decomposition formula of the form $I=I'\cap J$ where $J$ is a monomial ideal, first show that the intersection $I'\cap J$ can be generated by products of minimal generators of $I'$ with monomials (not necessarily from $J$), with the additional property that if a sum of such products is in  $I'\cap J$ then each of the summands is also in $I'\cap J$ (perhaps we should call such $I$ an \emph{$I'$- monomial ideal}).  Next fix a monomial $m$ and split the minimal generators $V(I')$ into two parts say $V(I')=V_m(I')\oplus V^m(I')$, determined by whether or not  $m$ appears in their monomial expansion or not.  In our situation, one can show that if $m$ is not in the support of $\nu\in V^m(I')$, then $m\cdot \nu\in I$, and, with more effort, one can also show that if $\nu\in V_m$ and $m\cdot\nu\in I'\cap J$ then $m\cdot \nu\in I$.  We highlight this argument here because it seems important in the theory of two-rowed Specht ideals.

This paper is organized as follows.  In Section \ref{sec:Specht} we define Specht polynomials, shifted Specht polynomials, and the modules and ideals they generate.  We then prove Theorem \ref{thm:A} and compute the dimensions of our shifted Specht modules  (Theorem \ref{thm:basis} and Theorem \ref{thm:dimMinGenSpecht}).  In Section \ref{sec:Lefschetz} we draw out the connection between Lefschetz properties and Specht modules, decompose the shifted Specht module into its irreducible representations, and derive other useful consequences.  In Section \ref{sec:radical} we prove Theorem \ref{thm:B} and Theorem \ref{thm:B1} (Theorem \ref{thm:rad} and Theorem \ref{thm:radD}).  In Section \ref{sec:perfection} we prove Theorem \ref{thm:C}(4.) (Theorem \ref{thm:perfect}), Theorem \ref{thm:SpechtP} (Theorem \ref{thm:perfectSpecht}), and Theorem \ref{thm:C}(3.) (Theorem \ref{thm:perfectD}), in that order.

\section*{Acknowledgments}
We would like to thank Larry Smith and University of G\"ottingen for their hospitality during our visit in the spring semester of 2019, when our discussions of Specht ideals and principal radical systems first began.  We would also like to thank Larry Smith for pointing out that the idea of the principal radical system was first used by Kiyoshi Oka.
The first author thanks Endicott College for the time and resources necessary to work on this project.   
The second author thanks K.\@ Yanagawa for the introduction to Specht ideals and for his insightful problem.

\tableofcontents

\section{Shifted Specht Polynomials, Modules, and Ideals}
\label{sec:Specht}
Fix a positive integer $n$, a field $\F$, and let $R=\F[x_1,\ldots,x_n]$ be the standard graded polynomial ring in $n$-variables, equipped with the usual action of $\mathfrak{S}_n$ which permutes the variables.  A partition of $n$, denoted $\lambda\vdash n$ is a sequence of non-increasing integers which sum to $n$, i.e. $\lambda=(\lambda_1,\ldots,\lambda_r)$ where $\lambda_1\geq\cdots\geq \lambda_r$ and $\sum_{i=1}^r\lambda_i=n$.  The Young diagram of $\lambda$ is a left-justified array of boxes with $r$-rows and $\lambda_i$ boxes in each row, and a filling of those boxes with distinct numbers $1,\ldots n$ is called a tableau of shape $\lambda$; the set of all tableaux of shape $\lambda$ will be written as $\operatorname{Tab}(\lambda)$.  To each tableau $T\in\tab(\lambda)$ of shape $\lambda$ we associate a polynomial $F_T\in R$ as follows.

First for any subset $S\subset\{1,\ldots,n\}$, define the $S$-Vandermonde polynomial 
$$\Delta_S=\prod_{i<j\in S}(x_i-x_j)$$
with the convention that if $|S|<2$ then $\Delta_S=1$.
Then for any tableau $T$ of shape $\lambda$ with columns $C_1,\ldots,C_{\lambda_1}$, define its Specht polynomial by 
$$F_T=\prod_{i=1}^{\lambda_1}\Delta_{C_i}.$$

The $\F$-linear span of Specht polynomials over $\tab(\lambda)$ is an $\mathfrak{S}_n$-representation called the Specht module, which we denote $V(\lambda)$, i.e.
$$V(\lambda)=\left\langle F_T \ | \ T\in\tab(\lambda)\right\rangle \operatorname{sp}_\F(F_T \ | \ T\in\tab(\lambda));$$
it is well known, in the case $\operatorname{char}(\F)=0$, that $V(\lambda)$ is irreducible, and conversely that every irreducible $\mathfrak{S}_n$-representation is isomorphic to $V(\lambda)$ for some $\lambda\vdash n$, e.g. \cite{Sagan}.  The Specht ideal of $\lambda$ is the ideal in $R$ generated by $V(\lambda)$, i.e.
$$\mathfrak{a}(\lambda)=V(\lambda)\cdot R=\left(F_T \ | \ T\in\tab(\lambda)\right).$$

In his paper \cite{Yana}, Yanagawa has formulated the following conjecture:
\begin{conjecture*}[Yanagawa '19]
	Over any field $\F$, and for any partition $\lambda$, the Specht ideal $\mathfrak{a}(\lambda)$ is radical.
\end{conjecture*}
This conjecture has been proved for partitions of the form $\lambda=(n-k,1,\ldots,1)$ by Yanagawa-Watanabe \cite{WY}, and for partitions of the form $\lambda=(n-k,k)$ and $\lambda=(d,d,1)$ by Yanagawa \cite{Yana}.  However, even in these simple cases, Yanagawa's proof that $\mathfrak{a}(\lambda)$ is radical is by no means easy.  This paper grew out of an attempt to understand and perhaps simplify Yanagawa's proof in the case of two rowed partitions $\lambda=(n-k,k)$, which we discuss next.

%\subsection{Shifted Specht Polynomials, Modules and Ideals}
Fix an integer $k$ satisfying $1\leq k\leq n-k$, and let $\lambda=(n-k,k)$ be the corresponding partition, which we regard as a left-justified two-rowed Young diagram with $n-k$ boxes in the first row and $k$ boxes in the second.  For each integer $d$ satisfying $1\leq k\leq d\leq n-k$, define the $d$-shifted shape $\lambda(d)$ to be the Young diagram obtained by moving the righter-most $n-d-k$ boxes on the first row to the left of the first boxes in the second row:
$$\lambda=\overbrace{\ydiagram{9,3}}^{n-k}\rightarrow \underbrace{\ydiagram{0,3}}_{n-k-d}\overbrace{\ydiagram{6,3}}^d=\lambda(d).$$
Note that $\lambda(k)$ is $\lambda$ rotated by $180^\circ$.

A tableau $T$ on shape $\lambda(d)$ is a labeling of the boxes of the Young diagram of $\lambda(d)$ with the numbers $\{1,\ldots,n\}$; we say that $T$ is standard if the rows are increasing from left to right, and the columns are increasing from top to bottom.  The set of tableaux on $\lambda(d)$ we denote by $\tab(n,k,d)$, and the set of standard tableaux by $\stab(n,k,d)$.

Given a tableau $T\in\tab(n,k,d)$ on the $d$-shifted shape $\lambda(d)$, such as
\begin{equation}
\label{eq:Tnkd}
T=\begin{ytableau}
\none & \none & \none & i_1 & \cdots & i_k & i_{k+1} & \cdots & i_d\\
i_{d+1} & \cdots & i_{n-k} & j_1 & \cdots & j_k\\
\end{ytableau}
\end{equation}
define the associated $d$-shifted Specht polynomial to be the homogeneous polynomial of degree $d$ by 
$$F_T(d)=(x_{i_1}-x_{j_1})\cdots (x_{i_k}-x_{j_k})\cdot x_{i_{k+1}}\cdots x_{i_d}.$$
Note that if $d=k$, then we recover the usual Specht polynomial:
$$F_T(k)=(x_{i_1}-x_{j_1})\cdots (x_{i_k}-x_{j_k});$$
we sometimes use the alternative notation $F_T$ or $F_T^k$ to remember the degree.  We allow $k=0$, and in this case a tableau $S\in\tab(n,0,d)$ has the form
\begin{equation}
\label{eq:n0d}
S=\begin{ytableau}
\none & \none & \none & \none & \none & \none & i_{1} & \cdots & i_d\\
i_{d+1} & \cdots & i_{n-k} & j_1 & \cdots & j_k\\
\end{ytableau}
\end{equation}
and its associated shifted Specht polynomial is the monomial 
$$M_S=x_{i_1}\cdots x_{i_d},$$
which we sometimes write as $M_S^d$ to remember degree.

The \emph{$d$-shifted Specht module of $\lambda=(n-k,k)$}, denoted by $V(n,k,d)$, is defined to be the $\F$-linear span of the $d$-shifted Specht polynomials, i.e. 
$$V(n,k,d)=\left\langle F_T(d) \ | \ T\in\tab(n,k,d)\right\rangle;$$
like the Specht module, it is also an $\mathfrak{S}_n$-representation, although it is not irreducible for $d>k$, cf. Section \ref{sec:Lefschetz}.  
The \emph{$d$-shifted Specht ideal of $\lambda=(n-k,k)$}, denoted by $\mathfrak{a}(n,k,d)$, is the ideal in $R$ generated by the shifted Specht module, i.e.
$$\mathfrak{a}(n,k,d)=V(n,k,d)\cdot R=\left(F_T(d) \ | \ T\in\tab(n,k,d)\right).$$

%The symmetric group $\mathfrak{S}_n$ acts on $R$ by permuting the variables, and its graded components are finite dimensional representations.  In case $\F=\C$ and $d=k$, the Specht modules $V(n,k,k)$ are irreducible, and conversely every irreducible $\mathfrak{S}_n$-representation is isomorphic to some Specht module.  On the other hand, the shifted Specht modules will not be irreducible in general.  
The remainder of the section will be devoted to finding a basis for the shifted Specht module $V(n,k,d)$ and to computing its dimension.

\subsection{Basis of a Shifted Specht Module}
\begin{theorem}
	\label{thm:basis}
	A basis for the $d$-shifted Specht module $V(n,k,d)$, and hence a minimal generating set of the ideal $\mathfrak{a}(n,k,d)$, are indexed by the standard tableaux on the $d$-shifted shape $\lambda(d)$, i.e.
	$$\left\{F_T(d) \ | \ T\in\stab(\lambda(d))\right\}.$$
\end{theorem}
The proof of Theorem \ref{thm:basis} comes in two steps, which we state as Lemmas. 

\begin{lemma}
	\label{lem:independent}
	The set $\left\{F_T(d)|T\in\stab(n,k,d)\right\}$ is linearly independent.
\end{lemma}
\begin{proof}
	By induction on $n\geq 2$.   The base case, where $n=2$ and $k=d=n-k=1$ is trivial.  For the inductive step, we assume that for every choice of $k'$ and $d'$ satisfying $1\leq k'\leq d'\leq (n-1)-k'$, the set $\left\{F_{T'}(d')|T'\in\stab(n-1,k',d')\right\}$ is linearly independent.  Fix any integers $1\leq k\leq d\leq n-k$ and suppose we have a dependence relation
	\begin{equation}
	\label{eq:dependence}
	\sum_{T\in\stab(n,k,d)}c_TF_T(d)=0.
	\end{equation}
	Note that for every $d$-standard tableau $T\in\stab(n,k,d)$ as in \eqref{eq:Tnkd}, we must have either $n=i_d$ or $n=j_k$; let $I_d\subset\stab(n,k,d)$ denote the set of $d$-standard tableaux with $n=i_d$ and let $J_d\subset \stab(n,k,d)$ be the ones with $n=j_k$.  Let $\pi\colon R\rightarrow S=\F[x_1,\ldots,x_{n-1}]$ be the projection sending $x_n$ to $0$, and note that for every $T\in I_d$ we have $\pi(F_T(d))=0$.  Moreover for every $T\in J_d$ we have $\pi(F_T(d))=F_{T'}(d)$ where 
	$$T'=\begin{ytableau}
	\none & \none & \none & i_1 & \cdots & i_{k-1} & i_k & i_{k+1} & \cdots & i_d\\
	i_{d+1} & \cdots &  i_{n-k} & j_1 & \cdots & j_{k-1}\\ 
	\end{ytableau}$$
	We see that $T'\in\stab(n-1,k-1,d)$, hence the induction hypothesis applies.
	Note that since the map $J_d\rightarrow \stab(n-1,k-1,d)$, sending $T\mapsto T'$ is one-to-one, and since 
	$$\pi\left(\sum_{T\in \stab(n,k,d)=I_d\sqcup J_d}c_TF_T(d)\right)=\sum_{T \in J_d}c_TF_{T'}(d)=0,$$
	we deduce, by the induction hypothesis, that $c_T=0$ for all $T\in J_d$.
	Then our dependence relation \eqref{eq:dependence} becomes 
	$$\sum_{T \in I_d}c_TF_T(d)=0=x_n\cdot\left(\sum_{T \in I_d}c_TF_{T''}(d)\right)$$
	where 
	$$T''=\begin{ytableau}
	\none & \none & \none & i_1 & \cdots & i_k & i_{k+1} & \cdots & i_{d-1}\\
	i_{d+1} & \cdots & i_{n-k} & j_1 & \cdots & j_k\\
	\end{ytableau}$$
	In this case we see that $T''\in\stab(n-1,k,d-1)$, and again our induction hypothesis applies.  Since the map $I_d\rightarrow \operatorname{STab}(n-1,k,d-1)$ sending $T\mapsto T''$ is one-to-one, and since 
	$$\sum_{T \in I_d}c_TF_{T''}(d)=0$$
	the induction hypothesis implies that $c_T=0$ for every $T\in I_d$ too, and therefore the dependence relation \eqref{eq:dependence} must be trivial.
\end{proof}

To see that the $d$-standard polynomials span $V(n,k,d)$ is a little more work.  We will follow the general method described in \cite[Section 2.6]{Sagan}.  For $T$ as in \eqref{eq:Tnkd} and any index $1\leq a\leq n$, define \emph{$a$-composition vector} to be the integer vector with $n-k$ components defined by 
$$\gamma^a(T)=(\gamma^a_1(T),\ldots,\gamma^a_{n-k}(T)), \ \ \text{where} \ \ \gamma^a_b(T)=\#\left\{c\in\operatorname{col}_b(T) \ | \ c\leq a\right\}.$$
Note that for a two rowed partition $\lambda=(n-k,k)$ the $a$-composition vector has entries $0$, $1$, or $2$.  Define the \emph{composition series} for $T$ to be the $n$-tuple of composition vectors $\gamma(T)=(\gamma^1(T),\ldots,\gamma^n(T))$; we  regard $\gamma(T)$ as a matrix whose columns are the composition vectors of $T$.  Given two vectors $v=(v_1,\ldots,v_{n-k})$ and $w=(w_1,\ldots,w_{n-k})$, we say that \emph{$w$ dominates $v$}, and write $v\triangleleft w$, if $v_1+\cdots+v_p\leq w_1+\cdots+w_p$ for all $1\leq p\leq n-k$.  Finally given two tableaux $T,T'\in\tab(\lambda(d))$, we say that $T'$ \emph{dominates} $T$, and write $T\triangleleft T'$, if every composition vector of $T'$ dominates the corresponding composition vector of $T$, i.e.
$$T\triangleleft T' \ \ \Leftrightarrow \ \ \gamma^a(T)\triangleleft \gamma^a(T'), \ \forall \ 1\leq a\leq n.$$
This composition-dominance order is a partial order on the set of tableaux $\tab(\lambda(d))$.  Moreover it is clear that the largest tableau is the one that fills the columns in order from left to right.  For example, if $n=5$, $k=1$, and $d=3$ the largest tableau (with increasing columns) is the standard tableau 
$$T=\begin{ytableau}
\none & 2 & 4 & 5\\
1 & 3\\
\end{ytableau}$$
with composition series 
$$\gamma(T)=\left(\begin{array}{lllll} 1 & 1 & 1 & 1 & 1\\
0 & 1 & 2 & 2 & 2\\
0 & 0 & 0 & 1 & 1\\
0 & 0 & 0 & 0 & 1\\
\end{array}\right)$$ 

Note that if $T$ and $T'$ have the same columns (possibly in different orders) then they have the same composition series, and they also have the same shifted Specht polynomials (up to sign), i.e. $\gamma(T)=\gamma(T')$ and $F_T(d)=\pm F_{T'}(d)$.  
The following lemma is useful for telling when one tableau dominates another.
\begin{lemma}
	\label{lem:dominance}
	If $1\leq a<b\leq n$ and $a$ appears in a column to the right of $b$, then 
	$$T\triangleleft (a,b)\cdot T$$
	where $(a,b)\cdot T$ is the tableau obtained from $T$ by transposing $a$ and $b$.   
\end{lemma}
\begin{proof}
	Note that for $1\leq i\leq a-1$ and for $b\leq i\leq n$ we have $\gamma^i(T)=\gamma^i((a,b)\cdot T)$.  Assume then that $a\leq i\leq b-1$, and suppose that $a$ and $b$ belong to columns $r$ and $q$ in $T$, respectively.  Then 
	$$\gamma^i((a,b)\cdot T) =\gamma^i(T)\ \begin{array}{c} \text{with $r^{th}$ part decreased by $1$}\\ \text{and $q^{th}$ part increased by $1$.}\\ \end{array}$$
	and since we are assuming that $q<r$, it follows that $\gamma^i(T)\triangleleft\gamma^i((a,b)\cdot T)$, and the result follows.
\end{proof}
We are now in a position to show the standard shifted Specht polynomials span the shifted Specht module.
\begin{lemma}
	\label{lem:span}
	The set $\left\{F_T(d)|T\in\stab(n,k,d)\right\}$ spans $V(n,k,d)$.
\end{lemma}
\begin{proof}
	We will show by downward induction on the composition-dominance order on  $\tab(\lambda(d))$ that for every $T\in\tab(\lambda(d))$, $F_T(d)$ can be written as a linear combination of shifted Specht polynomials associated to $d$-standard tableaux.  For the base case, note, as above, that the largest tableau of shifted shape $\lambda(d)$ is already standard.  Inductively, fix a shifted tableau $T\in\tab(\lambda(d))$ as in \eqref{eq:Tnkd}, and assume that for every tableau $T'$ that dominates $T$, $F_{T'}(d)$ can be written as a linear combination of shifted Specht polynomials corresponding to standard tableaux.  We may assume that the columns of $T$ are increasing.  If $T$ has no row descents, then $T$ must be standard and we are done.  Otherwise $T$ has some row descent, say between the $a^{th}$ and $a+1^{st}$ column.  There are several cases to consider and we claim that in all cases we can write $F_T(d)$ as a linear combination 
	$$F_T(d)=\sum_{T \triangleleft T'}c_{T'}F_{T'}(d).$$
	
	\noindent\textbf{Case 1:}  $1\leq a\leq n-k-d-1$.  Set $b=d+a$; in this case, we can merely swap $i_b$ and $i_{b+1}$ without affecting $F_T(d)$; in other words setting $T'=(i_b,i_{b+1})\cdot T$ we have 
	$$F_T(d)=F_{T'}(d)$$
	and since $T\triangleleft (i_b,i_{b+1})\cdot T$, it follows from our inductive hypothesis that $F_T(d)$ can be written as a linear combination of $d$-standard Specht polynomials on the shifted shape $\lambda(d)$.
	
	\noindent\textbf{Case 2:}  $a=n-k-d$.  Then we have 
	$$T=\begin{ytableau}
	\none & \none & \none & i_1 & \cdots & i_k & i_{k+1} & \cdots & i_d\\
	i_{d+1} & \cdots & i_{n-k} & j_1 & \cdots & j_k\\
	\end{ytableau}$$
	and we have the descent $i_{n-k}>j_1>i_1$.  Then by lemma \ref{lem:dominance} we see that 
	$$T\triangleleft T'=\left(i_{n-k},j_{1}\right)\cdot T \ \text{and} \ T\triangleleft T''=\left(i_1,i_{n-k}\right)\cdot T.$$
	Moreover one can easily check that 
	$$F_T(d)=F_{T'}(d)-F_{T''}(d)=G\left((x_{i_1}-x_{i_{n-k}})-(x_{j_1}-x_{i_{n-k}})\right)$$ 
	and inductive hypothesis applies.
	
	\noindent\textbf{Case 3:} $n-k-d+1\leq a\leq n-d-1$.  Set $b=a-n+k+d$ so we have 
	$$T=\begin{ytableau}
	\none & \none & \none & i_1 & \cdots & i_b & i_{b+1} & \cdots & i_k & i_{k+1} & \cdots & i_d\\
	i_{d+1} & \cdots & i_{n-k} & j_1 & \cdots & j_b & j_{b+1} & \cdots & j_k\\
	\end{ytableau}$$
	
	\textbf{Sub-Case 3a:}  $i_{b+1}<i_b<j_b$.  Then 
	$$T\triangleleft T'=\left(i_{b+1},j_b\right)\cdot T \ \text{and} \ T\triangleleft T''=\left(i_{b+1},i_b\right)\cdot T$$
	and again one can easily check that 
	$$F_T(d)=F_{T'}(d)+F_{T''}(d)=
	G\cdot\left((x_{i_b}-x_{i_{b+1}})(x_{j_{b}}-x_{j_{b+1}})+(x_{i_{b+1}}-x_{j_b})(x_{i_b}-x_{j_{b+1}})\right).$$
	and our inductive hypothesis applies.
	
	\textbf{Sub-Case 3b:} $i_b<i_{b+1}<j_{b+1}<j_b$.  Then 
	$$T\triangleleft T'=\left(i_{b+1},j_b\right)\cdot T \ \text{and} \ T\triangleleft T''=\left(j_{b+1},j_b\right)\cdot T$$
	and again one can easily check that 
	$$F_T(d)=F_{T'}(d)+F_{T''}(d)=
	G\cdot\left((x_{i_b}-x_{i_{b+1}})(x_{j_{b}}-x_{j_{b+1}})+(x_{i_{b}}-x_{j_{b+1}})(x_{i_{b+1}}-x_{j_{b}})\right)$$
	and our inductive hypothesis applies.
	
	\noindent\textbf{Case 4:} $a=n-d$.  Then 
	$$T=\begin{ytableau}
	\none & \none & \none & i_1 & \cdots & i_k & i_{k+1} & \cdots & i_d\\
	i_{d+1} & \cdots & i_{n-k} & j_1 & \cdots & j_k\\
	\end{ytableau}$$
	and we have the descent $j_{k}>i_k>i_{k+1}$.  Then by Lemma \ref{lem:dominance} we have
	$$T\triangleleft T'=\left(i_{k+1},i_k\right)\cdot T \ \text{and} \ T\triangleleft T''=\left(i_{k+1},j_k\right)\cdot T.$$
	and again one can check that
	$$F_T(d)=F_{T'}(d)+F_{T''}(d)=G\left(x_{i_k}\left(x_{i_{k+1}} -x_{j_k}\right) +x_{j_k}\left(x_{i_k}-x_{i_{k+1}}\right)\right)$$
	to which the inductive hypothesis once again implies.

	\noindent\textbf{Case 5:} $n-d+1\leq a\leq n-k-1$.  In this case set $b=a-(n-k-d)$ so that we have 
	$$T=\begin{ytableau}
	\none & \none & \none & i_1 & \cdots & i_k & i_{k+1} & \cdots & i_b & i_{b+1} &  \cdots & i_d\\
	i_{d+1} & \cdots & i_{n-k} & j_1 & \cdots & j_k\\
	\end{ytableau}$$
	with the descent $i_b>i_{b+1}$.  Hence by Lemma \ref{lem:dominance} we have 
	$$T\triangleleft T'=\left(i_{b+1},i_b\right)\cdot T$$
	and in this case we clearly have 
	$$F_T(d)=F_{T'}(d)$$
	to which the induction hypothesis applies again.
	
	Therefore in all cases we have shown that $F_T(d)$ is a linear combination of shifted Specht polynomials indexed by tableaux on the shifted shape $\lambda(d)$ which dominate $T$.  Therefore by induction, the $d$-standard shifted Specht polynomials  
	$$\left\{F_T(d)|T\in\stab(\lambda(d))\right\}$$
	must span $V(n,k,d)$.
\end{proof}

\begin{proof}[Proof of Theorem \ref{thm:basis}]
	By Lemma \ref{lem:independent}, the polynomials in the set 
	$$\left\{F_T(d) \ | \ T\in\stab(\lambda(d))\right\}$$
	are linearly independent, and by Lemma \ref{lem:span} they generate the shifted Specht module $V(n,k,d)$, and hence they must form a basis.
\end{proof}

\begin{example}
	\label{ex:513}
	Let $n=5$, $k=1$, and $d=3$.  There are $9$ standard tableau of shifted shape $\lambda(3)$ where $\lambda=(4,1)$:
	$$\begin{ytableau}
	\none & 2 & 4 & 5\\
	1 & 3\\ 
	\end{ytableau} \ \ 
	\begin{ytableau}
	\none & 2 & 3 & 5\\
	1 & 4\\ 
	\end{ytableau} \ \ 
	\begin{ytableau}
	\none & 2 & 3 & 4\\
	1 & 5\\ 
	\end{ytableau} \ \ 
	\begin{ytableau}
	\none & 1 & 4 & 5\\
	2 & 3\\ 
	\end{ytableau} 
	$$
	
	$$\begin{ytableau}
	\none & 1 & 3 & 5\\
	2 & 4\\ 
	\end{ytableau} \ \ 
	\begin{ytableau}
	\none & 1 & 3 & 4\\
	2 & 5\\ 
	\end{ytableau} \ \
	\begin{ytableau}
	\none & 1 & 2 & 5\\
	3 & 4\\ 
	\end{ytableau} \ \
	\begin{ytableau}
	\none & 1 & 2 & 4\\
	3 & 5\\ 
	\end{ytableau} \ \
	\begin{ytableau}
	\none & 1 & 2 & 3\\
	4 & 5\\ 
	\end{ytableau}$$
	Hence a minimal generating set for the ideal $\mathfrak{a}(5,1,3)$, and a basis for the representation $V(5,1,3)$, is given by 
	$$\left\{\begin{array}{ccc}
	(x_2-x_3)x_4x_5, & (x_2-x_4)x_3x_5, & (x_2-x_5)x_3x_4,\\
	(x_1-x_3)x_4x_5, & (x_1-x_4)x_3x_5, & (x_1-x_5)x_3x_4,\\
	(x_1-x_4)x_2x_5, & (x_1-x_5)x_2x_4, & (x_1-x_5)x_2x_3\end{array}\right\}.$$ 
\end{example}

The utility of Theorem \ref{thm:basis} is that it can transform set maps on the set of standard tableau $\stab(\lambda(d))$ to linear maps on the shifted Specht module $V(n,k,d)$, and sometimes the shifted Specht ideal $\mathfrak{a}(n,k,d)$.  The following useful corollaries illustrate this point. 

First note every standard tableau $T\in\stab(n,k,k)$ has $n$ in its support; in fact it must have the form
$$T=\begin{ytableau}
\none & \none & \none & i_1 & \cdots & i_{k-1} & i_k\\
i_{k+1} & \cdots & i_{n-k} & j_1 & \cdots & j_{k-1} & n\\
\end{ytableau}$$
We get a standard tableau $T'\in \stab(n-1,k-1,k)$ from $T$ by deleting the box containing $n$, i.e.
$$T'=\begin{ytableau}
\none & \none & \none & i_1 & \cdots & i_{k-1} & i_k\\
i_{k+1} & \cdots & i_{n-k} & j_1 & \cdots & j_{k-1} \\
\end{ytableau}$$
Since this map is obviously a bijection it extends to a linear isomorphism 
$$V(n,k,k)\rightarrow V(n-1,k-1,k).$$
In fact it can be extended to a ring isomorphism of Specht ideals using a linear change of coordinates:  Define the variables $y_1,\ldots,y_n$ by the formula
$$y_i=\begin{cases} x_n-x_i & \text{if} \ 1\leq i\leq n-1\\
x_n & \text{if} \ i=n\\
\end{cases}$$
and let $\Phi\colon R\rightarrow R$ be the change of coordinates map $\Phi(x_i)=y_i$.
\begin{corollary}
	\label{cor:CoC}
	Fix integers $k,n$ satisfying $1\leq k\leq n-k$.  Then we have a ring isomorphism
	$$\mathfrak{a}(n,k,k)\cong \Phi\left(\mathfrak{a}(n,k,k)\right)=\mathfrak{a}(n-1,k-1,k).$$
	In fact, in adding the principal ideal $(x_n)$, this isomorphism becomes equality:
	$$\mathfrak{a}(n,k,k)+(x_n)=\mathfrak{a}(n-1,k-1,k)+(x_n).$$
\end{corollary}
\begin{proof}
	With $T\in \stab(n,k,k)$ and $T'\in\stab(n-1,k-1,k)$ as above, we compute 
	$$F_T=\prod_{t=1}^k(x_{i_t}-x_{j_t})=\prod_{t=1}^{k-1}(x_{i_t}-x_{j_t})\cdot (x_{i_k}-x_n)=\prod_{t=1}^{k-1}(y_{j_t}-y_{i_t})\cdot (-y_{i_k})=(-1)^k\cdot \Phi\left(F_{T'}\right)$$
	and hence by Theorem \ref{thm:basis}, $\Phi$ sends a the first statement follows.  To see the second statement note that the ring automorphism $\Phi\colon R\rightarrow R$ becomes the negative identity map on the quotient $\bar{\Phi}=-I\colon R/(x_n)\rightarrow R/(x_n)$, and hence 
	$$\mathfrak{a}(n,k,k)+(x_n)=\mathfrak{a}(n-1,k-1,k)+(x_n)$$
	which is the second statement.
\end{proof}

We say that an index $i$, $1\leq i\leq n$ is in the support of tableau $T\in\tab(n,k,d)$, and write $i\in\operatorname{supp}(T)$, if $i$ appears in or to the right of a column with more than one row, i.e. if 
$$T=\begin{ytableau}
\none & \none & \none & i_1 & \cdots & i_k & i_{k+1} & \cdots & i_d\\
i_{d+1} & \cdots & i_{n-k} & j_1 & \cdots & j_k\\
\end{ytableau}$$
then $\operatorname{supp}(T)=\{i_1,j_1,\cdots,i_k,j_k,i_{k+1},\ldots,i_d\}$.  Fix an integer $m$ satisfying $0\leq m\leq d$, and define the subset $\stab_m(n,k,d)\subset \stab(n,k,d)$ consisting of standard tableaux which contain $\{1,\ldots,m\}$ in its support.  A standard tableau $T\in\stab_m(n,k,d)$ necessarily has the form
$$T=\begin{ytableau}
\none & \none & \none & 1 & \cdots & m & {}_{m+1} & \cdots & k & {}_{k+1} & \cdots & d\\
i_{d+1} & \cdots & i_{n-k} & j_1 & \cdots & j_m & j_{m+1} & \cdots & j_k\\
\end{ytableau}$$
and we define its image tableau as the one obtained by removing the boxes containing the numbers $1,\ldots,m$:
$$T'=\begin{ytableau}
\none & \none & \none & \none & \none & \none & {}_{m+1} & \cdots & k & {}_{k+1} & \cdots & d\\
i_{d+1} & \cdots & i_{n-k} & j_1 & \cdots & j_m & j_{m+1} & \cdots & j_k\\
\end{ytableau}$$  
Note that $T'\in \stab([n]_m,k-m,d-m)$ where $[n]_m$ means the tableaux are filled with numbers $\{m+1,\ldots,n\}$, and we count $k-m$ as zero if $k\leq m$.  The map of sets $\stab_m(n,k,d)\ni T\mapsto T'\in \stab([n]_m,k-m,d-m)$ is evidently one-to-one and onto, and by Theorem \ref{thm:basis}, it induces an bijective linear map of Specht modules, which is the key to some of the main technical arguments in this paper.
\begin{corollary}
	\label{cor:inductive}
	The induced linear map 
	$$\xymatrixrowsep{.5pc}\xymatrix{V_m(n,k,d)\ar[r] & V([n]_m,k-m,d-m)\\ F_T(d)\ar@{|->}[r] & F_{T'}(d-m)}$$
	is bijective.
\end{corollary}

\subsection{Dimension of a Shifted Specht Module}
We compute the dimension of the shifted Specht module $V(n,k,d)$ by counting the standard tableaux on the shifted shape $\lambda(d)$.
As a first step, we observe that a standard tableau $T$ on $\lambda(d)$ is uniquely determined by its $d$-row, or its $(n-d)$-row, i.e. if 
$$T=\begin{ytableau} \none & \none & \none & i_1 & \cdots & i_k & i_{k+1} & \cdots & i_d\\ 
i_{d+1} & \cdots & i_{n-k} & j_1 & \cdots & j_k\\ \end{ytableau}$$ its $d$-row is the sequence $(i_1,\ldots,i_d)$, and its $(n-d)$-row is $(i_{d+1},\ldots,i_{n-k},j_1,\ldots,j_k)$.  
\begin{lemma}
	\label{lem:nskrow}
	A necessary and sufficient condition for an increasing sequence of $d$ integers in $[n]$, say $\{i_1<\cdots<i_d\}$, to be a $d$-row of a standard tableau of shape $\lambda(d)$ is that for each index $1\leq t\leq k$, we have 
	\begin{equation}
	\label{eq:condition}
	\#\left\{j\in[n]\setminus\{i_1,\ldots,i_d\} \ | \ j>i_{k+1-t}\right\}\geq t.
	\end{equation} 
\end{lemma}
\begin{proof}
	Assume that \eqref{eq:condition} holds.  Write the compliment $[n]\setminus \{i_1,\ldots,i_d\}$ sequence in increasing order as $i_{d+1}< \cdots<i_{n-k}<j_1<\cdots<j_k$.  Then Condition \eqref{eq:condition} implies that $i_t<j_t$ for each $1\leq t\leq k$, and hence the tableau with $d$-row $(i_1,\ldots,i_d)$ must be standard.
	Conversely, suppose that $T$ in \eqref{eq:condition} is standard.  Then for each $1\leq t\leq k$ we have 
	$i_t<j_t<j_{t+1}<\cdots <j_k$, which implies Condition \eqref{eq:condition}. 
\end{proof}

In fact, we can encode a tableau $T$, with increasing rows but not necessarily standard, on the shifted shape $\lambda(d)$ as a NE lattice path from $(0,0)$ to $(n-d,d)$.  Then the standard tableaux will correspond to NE lattice paths which stay below the shifted diagonal line $y=x+(d-k)$.  Here are some more precise statements.

\begin{definition}
	\label{def:NELP}
	Define a NE lattice path from $(a,b)$ to $(a',b')$ to be a sequence of points on the $\Z^2$ lattice, $P_0=(a,b),P_1,\ldots,P_{m}=(a',b')$ where if $P_i=(a_i,b_i)$ then either $P_{i+1}=(a_i+1,b_i)$ in which case the $i^{th}$ step is E, or $P_{i+1}=(a_i,b_i+1)$ in which case the $i^{th}$ step is N.
\end{definition}
The following fact is well known.
\begin{fact}
	\label{fact:NELPcount}
	The number of NE lattice paths from $(a,b)$ to $(a',b')$ is the binomial coefficient
	$$\binom{\overbrace{(a'+b')-(a+b)}^m}{b'-b}.$$
\end{fact}
%For the proof, note that in such a NE lattice path, among the $m=(a'+b')-(a+b)$ total steps, exactly $b'-b$ of those must be N-steps, and we may choose those freely.

There is a map from the set of NE lattice paths from $(0,0)$ to $(n-d,d)$ to the set of tableaux of shifted shape $\lambda(d)$ with increasing rows:  label the $i^{th}$ step $n-i+1$, i.e. 
\begin{equation}
\overline{P_iP_{i+1}}\mapsto n-i+1
\end{equation}
and fill in the shifted diagram $\lambda(d)$ with these numbers from right to left with the N-steps in the $k$-row and the E-steps in the $(n-k)$-row.

\begin{example}
	\label{ex:NELPmap}
	Let $n=11$, $k=4$, $d=6$, so that $n-d=5$.  Then setting $\lambda=(7,4)$, the NE lattice path from $(0,0)$ to $(n-d,d)=(5,6)$ given by
	\begin{center}
		\includegraphics[width=7cm]{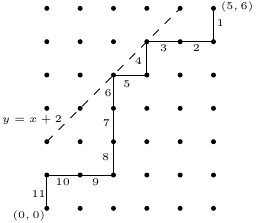}
	\end{center}
	corresponds to the tableau of shifted shape given by 
	$$\begin{ytableau}
	\none & 1 & 4 & 6 & 7 & 8 & 11\\
	2 & 3 & 5 & 9 & 10\\
	\end{ytableau}$$
	Note that the NE lattice path here is sub-shifted-diagonal, and that the Young tableau here is standard.  A non-sub-shifted-diagonal NE lattice path is given by 
	\begin{center}
		\includegraphics[width=7cm]{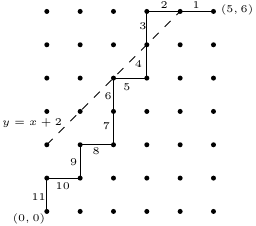}
	\end{center}	
	and corresponds to the non-standard Young tableau
	$$\begin{ytableau}
	\none & 3 & 4 & 6 & 7 & 9 & 11\\
	1 & 2 & 5 & 8 & 10\\
	\end{ytableau}$$
\end{example}  
The theorem, evidenced by Example \ref{ex:NELPmap}, is that the above map gives a bijection between sub-shifted-diagonal NE lattice paths and standard Young tableaux.

\begin{lemma}
	\label{lem:NELP}
	There is a bijection between NE lattice paths from $(0,0)$ to $(n-d,d)$ which do not cross the shifted diagonal line $y=x+(d-k)$ and standard tableaux of shifted shape $\lambda(d)$. 
\end{lemma}
\begin{proof}
	Translating condition \eqref{eq:condition} from Lemma \ref{lem:nskrow} into the language of NE lattice paths, it says:
	\begin{equation}
	\label{eq:NELPCondition}
	\# \ \text{E-steps to the left of the $(d-k+t)^{th}$ N-step} \geq t.	
	\end{equation}
	Hence if the $(d-k+t+1)^{th}$ $N$-step is the $p^{th}$ step in the path, then the point $P_p=(e_p,n_p)$ has coordinates that satisfy $n_p=(d-k+t)$ and $e_p\geq t$, and hence $n_p-e_p\leq d-k$.  This implies that the point $P_p$ lies below the shifted diagonal line $y=x+(d-k)$, as desired.
\end{proof}

Lemma \ref{lem:NELP} translates our problem of counting standard tableaux of shifted shape $\lambda(d)$ into the problem of counting sub-shifted-diagonal lattice paths from $(0,0)$ to $(n-d,d)$.  The advantage of counting lattice paths lies in the so-called reflection principal; a proof can be found in \cite{Feller}.
\begin{lemma}[Reflection Principle]
	\label{lem:rp}
	Let $\ell:y=x+c$ be any line with slope $1$, and suppose that $(a',b')$ and $(a,b)$ are two lattice points that lie on the same side of $\ell$.  Then the number of NE lattice paths from $(a',b')$ to $(a,b)$ which touch or cross $\ell$ is equal to the total number of NE lattice paths from $(a'',b'')$ to $(a,b)$, where $(a'',b'')$ is the reflection of $(a',b')$ over $\ell$.
\end{lemma}

We are now in a position to compute the dimension of the shifted Specht module $V(n,k,d)$.
\begin{theorem}
	\label{thm:dimMinGenSpecht}
	The number of NE lattice paths from $(0,0)$ to $(n-d,d)$ which do not cross the shifted diagonal $y=x+(d-k)$, and hence the number of standard tableaux on the shifted shape $\lambda(d)$, is equal to 
	$$\dim(V(n,k,d))=\binom{n}{d}-\binom{n}{k-1}.$$
\end{theorem}
\begin{proof}
	To count the number of shifted subdiagonal NE lattice paths from $(0,0)$ to $(n-d,d)$ is the same as to count the NE lattice paths which do not touch or cross the line $y=x+(d-k)+1$.
	%\begin{center}
	%	\includegraphics[width=10cm]{NELPpaper3}
	%\end{center}  
	Let us instead count the number of NE lattice paths from $(0,0)$ to $(n-d,d)$ that do touch or cross the line $y=x+(d-k)+1$, which, according to the Reflection Principle, is the same as the total number of NE lattice paths from $(-(d-k)-1,(d-k)+1)$ to $(n-d,d)$.  According to Fact \ref{fact:NELPcount}, this number is 
	\begin{equation}
	\label{eq:count1}
	\binom{(n-d+d)+(-(d-k)-1+(d-k)+1)}{d-(d-k)-1}=\binom{n}{k-1}.
	\end{equation} 
	We can also count the total number of NE lattice paths from $(0,0)$ to $(n-d,d)$ as 
	\begin{equation}
	\label{eq:count2}
	\binom{(n-d+d)-(0+0)}{d-0}=\binom{n}{d}.
	\end{equation} 
	We therefore get the number of NE lattice paths from $(0,0)$ to $(n-d,d)$ which do not cross the line $y=x+(d-k)$ by subtracting \eqref{eq:count1} from \eqref{eq:count2}, which gives the desired formula.
\end{proof}

\section{Lefschetz Propoerties}
\label{sec:Lefschetz}
Let $E\subset R$ be the graded vector subspace spanned by square-free monomials in the variables $x_1,\ldots,x_n$.  It will be convenient to identify $E$ with a quotient of $R$ as well, specifically the Artinian monomial complete intersection 
$$A=\frac{\F[x_1,\ldots,x_n]}{(x_1^2,\ldots,x_n^2)}.$$
Define the following linear operators on $A$:  The \emph{raising operator} is multiplication by the sum of variables:
$$L=\times (x_1+\cdots+x_n)\colon A\rightarrow A[1],$$
the \emph{lowering operator} is the partial derivative map corresponding to the sum of variables:
$$D=\frac{\partial}{\partial x_1} + \cdots +\frac{\partial}{\partial x_n}\colon A\rightarrow A[-1],$$
and the \emph{semi-simple operator}:
$$H\colon A\rightarrow A, \ \ H(a)=(n-2k)\cdot a, \ \forall \ a\in A_k.$$
\begin{lemma}
	\label{lem:SL2}
	The operators $\{D,L,H\}$ forms an $\mathfrak{sl}_2$-triple on $E\cong A$, i.e. they satisfy the commutator relations:
	$$[D,L]=H, \ \ [H,D]=2D, \ \ [H,L]=-2L.$$
\end{lemma} 
\begin{proof}
	Since $D$, $L$, and $H$ are linear, it suffices to check the relations on monomials, and by symmetry it suffices to check only the single square-free monomial $\mu=x_1\cdots x_i$.  We have 
	\begin{align}
	\nonumber D\circ L(m)= & D\left(\sum_{j=i+1}^nx_1\cdots x_ix_{j}\right)= \sum_{j=i+1}^nD(x_1\cdots x_i\cdot x_j)= \sum_{j=i+1}^n\left(x_1\cdots x_i+\sum_{k=1}^ix_1\cdots\hat{x}_k\cdots x_i\cdot x_j\right)\\
	\label{eq:DLm}
	= & (n-i)\cdot m+\sum_{j=i+1}^n\sum_{k=1}^ix_1\cdots \hat{x}_k\cdots x_i x_j\\
	\nonumber L\circ D(m)= & L\left(\sum_{k=1}^ix_1\cdots \hat{x}_k\cdots x_i\right)=\sum_{k=1}^i\left(L(x_1\cdots\hat{x}_k\cdots x_i)\right)=\sum_{k=1}^i\left(x_1\cdots x_i+\sum_{j=i+1}^nx_1\cdots\hat{x}_k\cdots x_i\cdot x_j\right)\\
	\label{eq:LDm}
	= & i\cdot m+\sum_{k=1}^i\sum_{j=i+1}^nx_1\cdots \hat{x}_k\cdots x_i\cdot x_j.
	\end{align} 
	Subracting \eqref{eq:DLm} and \eqref{eq:LDm} yields $[D,L](m)=D\circ L(m)-L\circ D(m)=(n-2i)\cdot m= H(m)$, and hence verifies the relation $[D,L]=H$.  Verifications of the other two relations are straightforward and left to the reader. 
\end{proof}
For each $0\leq i\leq n$, define the $i^{th}$-primitive subspace $P_i\subset A_i$ by
$$P_i=\ker(D)\cap A_i=\left\{\alpha\in A_i \ | \ D(\alpha)=0\right\}.$$
It follows from Lemma \ref{lem:SL2} that for any positive integer $m$ and for any $\alpha\in P_k$ we have 
$$D\left(L^m(\alpha)\right)=m\cdot (n-2k+1-m)\cdot L^{m-1}(\alpha).$$
Also note that for any Specht polynomial $F_T\in V(n,k,k)\subset A_k$, we have 
$$D\left(F_T\right)=0.$$
In particular, we have a chain of containments
$$V(n,k,k)\subseteq P_k\subseteq \ker(L^{n-2k+1})\cap A_k.$$
\begin{lemma}
	\label{lem:Dsurj}
	Equality $V(n,k,k)=P_k$ holds if and only if the derivative map, i.e. the lowering operator 
	$$D\colon A_k\rightarrow A_{k-1}$$
	is surjective.
\end{lemma}
\begin{proof}
	Assume that $V(n,k,k)=P_k$.  Then by Theorem \ref{thm:dimMinGenSpecht} we have 
	$$\dim(P_k)=\binom{n}{k}-\binom{n}{k-1}=\dim(A_k)-\dim(A_{k-1}).$$
	Let $I_{k-1}=D(A_k)$ be the image of the derivative map.  By linear algebra $\dim(I_{k-1})+\dim(P_k)=\dim(A_k)$ hence $\dim(I_{k-1})-\dim(A_{k-1})=0$, hence the derivative map is surjective.  Conversely, if $D\colon A_k\rightarrow A_{k-1}$ is surjective, then $\dim(I_{k-1})=\dim(A_{k-1})$, hence $\dim(P_k)=\dim(A_k)-\dim(A_{k-1})=\dim(V(n,k,k))$.  Since $V(n,k,k)\subseteq P_k$, and they have the same dimension, this containment must be equality. 
\end{proof}
As we shall see, surjectivity of the derivative map in Lemma \ref{lem:Dsurj} is dictated by the weak Lefschetz property.
\subsection{Weak Lefschetz Property}
For an arbitrary graded Artinian algebra $C=R/I$, we say that $C$ has the \emph{weak Lefschetz property} if there is a linear form $\ell\in C_1$ such that the multiplication maps 
\begin{equation}
\label{eq:WL}
\times\ell\colon C_{i-1}\rightarrow C_{i}
\end{equation}
have maximum rank for every degree $i\geq 0$; in this case we call $\ell$ a weak Lefschetz element for $C$.  If $C$ has a unimodal Hilbert function with socle degree $d$, then $\ell\in C_1$ is Lefschetz if and only if the multiplication maps \eqref{eq:WL} are injective for $1\leq i\leq \flo{d+1}$ and surjective for $\flo{d+3}\leq i\leq d$.  If $C$ is Gorenstein, then it suffices only to check that \eqref{eq:WL} is injective in degrees $1\leq i\leq \flo{d+1}$.  In fact one can show that if $C$ is Gorenstein with the standard grading and if the multiplication map \eqref{eq:WL} is injective for some $i_0$, then it is injective for all $i\leq i_0$.  Moreover if the ideal $I$ is generated by monomials then $C$ is weak Lefschetz if and only if $x_1+\cdots+x_n\in C_1$ is a weak Lefschetz element.  For more details, especially regarding these last two facts, see \cite[Propositions 2.1, 2.2]{MMN} or \cite[Proposition 2.5]{Wat87}.

In our situation $A$ is a standard graded Artinian Gorenstein algebra with unimodal Hilbert function and cut out by a monomial ideal.  In fact, in our situation, the matrix for the multiplication map $L\colon A_{k-1}\rightarrow A_{k}$ in the monomial basis is the transpose of the derivative map $D\colon A_{k}\rightarrow A_{k-1}$.  Therefore we see that $A$ has the weak Lefschetz property if and only if the derivative maps  
$$D\colon A_k\rightarrow A_{k-1}$$
are surjective for all $1\leq k\leq \flo{n+1}$.
The following result is due to Kustin-Vraciu \cite{KV}, and we refer the reader there for a proof.  As usual, $p=\operatorname{char}(\F)\geq 0$.
%We say that the pair $(A,L)$ is \emph{weak Lefschetz} if the restriction of $L$ to the graded components of $A$ always has maximal rank.  For Gorenstein algebras $(A,L)$ is weak Lefschetz if and only if the maps 
%$$L\colon A_k\rightarrow A_{k+1}$$
%are injective for all $0\leq k\leq \left\lfloor\frac{n+1}{2}\right\rfloor$.
%In our case note that the matrix for the map $L\colon A_k\rightarrow A_{k+1}$ in the monomial basis is the transpose of the map $D\colon A_{k}\rightarrow A_{k-1}$, hence $(A,L)$ is weak Lefschetz if and only if the maps
%$$D\colon A_{k}\rightarrow A_{k-1}$$
%are surjective for all $0\leq k\leq \left\lfloor\frac{n+1}{2}\right\rfloor$.

\begin{lemma}
	\label{lem:WLP}
	The monomial complete intersection $A$ has the weak Lefschetz property if and only if $p=0$ or $p\geq \flo{n+3}$.
\end{lemma}

%There is a small nuance here.  Note that if $n$ is odd, say $n=2m+1$ for some integer $m\geq 0$, then for $k=\left\lfloor\frac{n+1}{2}\right\rfloor=m+1$ we have $k=m+1>n-k=m$, indicating that the Specht module $V(n,k,k)$ is not defined.  In accordance with Lemma \ref{lem:WLP} we set $V(n,k,k)=0$ in this case.  In fact if we just want condition (3.) for $0\leq k\leq \flo{n}$, we can weaken the other conditions a bit.  To see this consider the subring $B\subset A$ defined by 
%$$B=\frac{\F[x_1,\ldots,x_{n-1}]}{(x_1^2,\ldots,x_{n-1}^2)}$$
%which comes with its own maps $\{D_B,L_B,H_B\}$.
From Lemma \ref{lem:WLP} we derive the following useful result.
\begin{lemma}
	\label{lem:Vnkk}
		Fix any integer $k$ satisfying $2k\leq n$.  Then the following are equivalent:
		\begin{enumerate}
			\item $p=0$ or $p\geq k+1$.
			\item The derivative maps $D\colon A_i\rightarrow A_{i-1}$
			are surjective for all $1\leq i\leq k$.
			\item The derivative map $D\colon A_k\rightarrow A_{k-1}$
			is surjective.
			
		\end{enumerate}
\end{lemma}
\begin{proof}
	(1.) $\Rightarrow$ (2.). Assume that $p=0$ or $p\geq k+1$.  For each index $j$, $2k\leq j\leq n$, define the nested chain of monomial complete intersections 
	$$
	C_{2k}\subset \cdots \subset C_n, \ \ \text{where} \ \ C_j=\frac{\F[x_1,\ldots,x_j]}{(x_1^2,\ldots,x_j^2)}
	$$
	By Lemma \ref{lem:WLP}, the first monomial complete intersection $C_{2k}$ 
	has the weak Lefschetz property, and in particular the derivative map 
	$$D_{C_{2k}}\colon \left(C_{2k}\right)_i\rightarrow \left(C_{2k}\right)_{i-1}$$
	is surjective for all $1\leq i\leq k$.  Inductively for $j>2k$, note that for each $0\leq i\leq k$ we have direct sum decomposition of graded vector spaces 
	$$\left(C_{j}\right)_i\cong \left(C_{j-1}\right)_i\oplus x_j\cdot \left(C_{j-1}\right)_{i-1}$$
	and in particular the derivative map in degree $k$ decomposes into block triangular form
	$$D_{C_j}\colon \left(C_j\right)_i\rightarrow\left(C_j\right)_{i-1}=
	\left(\begin{array}{ccr}  D_{C_{j-1}}\colon\left(C_{j-1}\right)_i\rightarrow\left(C_{j-1}\right)_{i-1}& \vline &  I\colon\left(C_{j-1}\right)_{i-1}\rightarrow \left(C_{j-1}\right)_{i-1}\\ 
	&\vline & \\
	\hline
	& \vline & \\
	0 & \vline & D_{C_{j-1}}\colon\left(C_{j-1}\right)_{i-1}\rightarrow\left(C_{j-1}\right)_{i-2}\\ \end{array}\right)$$
	Since the maps $$D_{C_{j-1}}\colon\left(C_{j-1}\right)_i\rightarrow\left(C_{j-1}\right)_{i-1}$$ 
	is surjective for all $1\leq i\leq k$, it follows that the maps $$D_{C_j}\colon\left(C_j\right)_i\rightarrow \left(C_j\right)_{i-1}$$
	is also surjective for all $1\leq i\leq k$.  In particular, this argument shows that the derivative maps $D\colon A_i\rightarrow A_{i-1}$ must be surjective for all $1\leq i\leq k$ as well.
	
	(2.) $\Rightarrow$ (3.) Obvious.
	
	(3.) $\Rightarrow$ (1.) Assume that $0<p<k+1$.  Then setting $i=p\leq k$, we claim that $D\colon A_p\colon A_{p-1}$ cannot be surjective.  Indeed, set $\alpha\in A_p$ to be the $p^{th}$-elementary symmetric polynomial in the first $2p-1\leq n-1$ variables: 
	$$\alpha=e_p(x_1,\ldots,x_{2p-1})\in A_p.$$
	Note first that over a field of characteristic $p$, we have 
	$$D(\alpha)=p\cdot e_{p-1}(x_1,\ldots,x_{2p-1})\equiv 0.$$
	On the other hand, we have 
	$$\alpha(1,\ldots,1)=\binom{2p-1}{p}=\frac{(2p-1)\cdots(p+1)}{(p-1)!}\neq 0.$$
	This shows that $\alpha\in P_p=\ker(D)\cap A_p$, but $\alpha\not\in V(n,p,p)$ (since every polynomial in $V(n,p,p)$ necessarily vanishes at any point in which at least $p+1$-entries are equal).  Therefore, by Lemma \ref{lem:Dsurj}, $D\colon A_p\rightarrow A_{p-1}$ is not surjective, and hence by \cite[Proposition 2.1]{MMN}, $D\colon A_k\rightarrow A_{k-1}$ cannot be surjective either.
\end{proof}

%{\color{red}Does the converse to Lemma \ref{lem:Vnkk} hold?  Example?}

We can also derive a useful result on specialization of Specht modules.  
\begin{lemma}
	\label{lem:restrict}
	Assume that $p=0$ or $p\geq k+1$, fix $j$ satisfying $2k\leq j\leq n$, and set 
	$$B=\frac{\F[x_1,\ldots,x_j]}{(x_1^2,\ldots,x_j^2)}.$$  Then 
	$$V(j,k,k)=V(n,k,k)\cap B_k.$$ 
\end{lemma}
\begin{proof}
	Containment $V(j,k,k)\subseteq V(n,k,k)\cap B_k$ is clear.  For the reverse containment, suppose that $\alpha\in V(n,k,k)\cap B_k$.  The key observation here is that the restriction of the derivative map on $A$, call it $D_A$, to the subspace $B\subset A$ is the same as $D_B$.  Since $\alpha\in V(n,k,k)\subseteq P_{A,k}=\ker(D_A)\cap A_k$, it follows that $D_A(\alpha)=0$, and hence also $D_B(\alpha)=0$.  This means that $\alpha\in P_{B,k}$.  From the proof of Lemma \ref{lem:Vnkk}, we deduce that $D_B\colon B_k\rightarrow B_{k-1}$ is surjective, which by Lemma \ref{lem:Dsurj} implies that $\alpha\in V(j,k,k)$, as desired.
\end{proof} 

%{\color{red}Does Lemma \ref{lem:restrict} hold without the weak Lefschetz condition?  Example?}
As we shall see, the weak Lefschetz property, or lack thereof, can be used to detect embedded primary components of our Specht-monomial ideals.  Next we shall use the strong Lefschetz property to decompose our shifted Specht modules into irreducible $\mathfrak{S}_n$-representations.
 
\subsection{Strong Lefschetz Property} 
The pair $(A,L)$ is \emph{strong Lefschetz} if the restriction of $L^i$ to the graded components of $A$ always has maximal rank, or equivalently if 
$$L^{n-2k}\colon A_k\rightarrow A_{n-k}$$
are isomorphisms for all $0\leq k\leq \left\lfloor\frac{n}{2}\right\rfloor$.  One can show that we have containment
$$V(n,k,d)\subseteq\ker\left(L^{n-k-d+1}\right)\cap A_k.$$
\begin{lemma}
	\label{lem:SLP}
	Let $p=\operatorname{char}(\F)$.  The following are equivalent.
	\begin{enumerate}
		\item $p=0$ or $p\geq n+1$,
		\item The pair $(A,L)$ is strong Lefschetz.
		\item For all integers $0\leq k\leq d\leq n-k$, the shifted Specht modules $V(n,k,d)\subset A$ satisfy
		$$V(n,k,d)=\bigoplus_{i=k}^dL^{d-i}(P_i)=\ker\left(L^{n-k-d+1}\right)\cap A_d.$$
	\end{enumerate}
\end{lemma}
\begin{proof}
	(1.) $\Rightarrow $ (2.).  This is due to Ikeda; see \cite[Proposition 3.66]{HMMNWW}.
	
	(2.) $\Rightarrow$ (1.).  Note that if $p\leq n$, then we must have $L^p=0$, and $(A,L)$ cannot be strong Lefschetz.
	
	(2.) $\Rightarrow$ (3.).  If $(A,L)$ is strong Lefschetz then the map 
	$$L^{n-k-d+1}\colon A_d\rightarrow A_{n-k+1}$$
	must have maximal rank, and hence we must have 
	$$\dim(\ker(L^{n-k-d+1})\cap A_d)=\dim(A_d)-\dim(A_{n-k+1})=\binom{n}{d}-\binom{n}{k-1}=\dim(V(n,k,d)).$$ 
	But since we already have containment $V(n,k,d)\subseteq \ker(L^{n-k-d+1})\cap A_d$ it must be equality.  It follows that $P_i=\ker(L^{n-2i+1})\cap A_i$, and hence we have the containment
	$$\bigoplus_{i=k}^dL^{d-i}(P_i)\subseteq \ker(L^{n-k-d+1})\cap A_d.$$
	Since $(A,L)$ is strong Lefschetz it follows that for each $i$, $L^{d-i}(P_i)\cong P_i$, and hence a simple dimension count reveals this containment must also be equality.
	
	(3.) $\Rightarrow$ (2.).  Assume (3.) holds, and assume that (2.) does not.  Fix integer $k$ satisfying $1\leq k\leq \left\lfloor\frac{n}{2}\right\rfloor$ and suppose that $\alpha\in \ker(L^{n-2k})\cap A_k$.  Then certainly $\alpha\in \ker(A^{n-2k+1})\cap A_k$ hence $\alpha\in P_k$ by (3.).  But also according to (3.) we have 
	$$A_{n-k}=V(n,0,n-k)=\bigoplus_{i=0}^{n-k}L^{n-k-i}(P_i).$$
	On the other hand if $L^{n-2k}(\alpha)=0$, then $\dim(L^{n-2k}(P_k))<P_k$, and hence we must have 
	$$\dim\left(A_{n-k}\right)=\sum_{i=0}^{n-k}\dim\left(L^{n-k-i}(P_i)\right)<\sum_{i=0}^{n-k}P_i=\binom{n}{k}$$
	which is a contradiction.  Therefore (2.) must hold after all.
\end{proof}

If any one of the conditions in Lemma \ref{lem:SLP} is satisfied, one can show that $L^{d-i}(P_i)\cong P_i\cong V(n,k,k)$, and hence in this case we get a decomposition of the shifted Specht module $V(n,k,d)$ into irreducible $\mathfrak{S}_n$-representations.
\begin{corollary}
	\label{cor:VnkdDecomp}
	Let $p=\operatorname{char}(\F)$, and assume that $p=0$ or $p\geq n+1$.  Then the primitive decomposition of the shifted Specht module is a decomposition into irreducible $\mathfrak{S}_n$-representations:  
	$$V(n,k,d)\cong \bigoplus_{i=k}^dL^{d-i}(P_i)\cong \bigoplus_{i=k}V(n,i,i)[d-i].$$
	Here a basis for the irreducible component $L^{d-i}(P_i)$ is 
	$$\left\{e_{d-i}(T^c)\cdot F_T \ | \ T\in\tab(n,i,i)\right\}$$
	where $e_{d-i}(T^c)$ is the $(d-i)^{th}$ elementary symmetric polynomial in the variables which are not in the support of $T$.
\end{corollary}

\section{Radical of Shifted Specht Ideals}
\label{sec:radical}
Our main idea is principal radical systems, based on the following basic facts from commutative algebra:
\begin{lemma}
	\label{lem:PRS}
	Let $I\subset R$ be a homogeneous ideal and $x\in R\setminus I$ be any homogeneous polynomial satisfying $(I:x)=(I:x^2)$.
	\begin{enumerate}
		\item If $(I:x)=I$ and if $I+(x)$ is radical, then $I$ is radical too.
		\item If $(I:x)\neq I$ and if $(I:x)$ and $I+(x)$ are both radical, then $I$ is radical too. 
	\end{enumerate}
\end{lemma}
\begin{proof}
	Note that $I\subset I+(x)$ and if $I+(x)$ is radical we also have 
	$$\sqrt{I}\subset I+(x).$$
	Hence for $g\in \sqrt{I}$ we can find $a\in I$ and $b\in R$ such that 
	$g=a+xb$.  Since $g\in\sqrt{I}$, there is some integer $N$ for which $g^N\in I$, and we have $g^N=a'+x^Nb^N$ for some other $a'\in I$.  Therefore $x^Nb^N\in I$ and hence $b^N\in (I:x^n)=(I:x)$ and therefore $b\in\sqrt{(I:x)}$.  If $(I:x)\neq I$ and $(I:x)$ is radical, then $xb\in I$ and hence $g\in I$, and we are done.  If $(I:x)=I$, then $b\in \sqrt{I}$, and hence we can find $a_1\in I$ and $b_1\in R$ for which $b=a_1+xb_1$.  Looking back to $g$, we have $g=(a+xa_1)+x^2b_1$.  Repeating this procedure a number $m$-times will yield $g=(a+xa_1+x^2a_2+\cdots+x^{m}a_{m})+x^{m+1}b_m$, which implies that 
	$$g\in \bigcap_{m=1}^\infty I+(x^m).$$
	Since $x$ is homogeneous, it follows that $\bigcap_{m=1}^\infty I+(x^m)=I$, and the result follows.
\end{proof} 

An easy application of principal radical systems is the radical of monomial ideal $(x_1,\ldots,x_n)^{(d)}$.
\begin{lemma}
	\label{lem:PRSmonomial}
	For each integer $d$ satisfying $1\leq d\leq n$ the monomial ideal 
	$$(x_1,\ldots,x_n)^{(d)}$$
	is radical.
\end{lemma}	
\begin{proof}
By induction on $n\geq 1$, the base case being trivial.  For the inductive step, assume that $(x_1,\ldots,x_{n-1})^{(e)}$ is radical for all $1\leq e\leq n-1$, and fix an integer $d$ satisfying $1\leq d\leq n$.  Set $I=(x_1,\ldots,x_n)^{(d)}$ and $x=x_n$.  If $d=n$ then $I=(x_1\cdots x_n)$ is principal, and clearly radical, hence we may assume that $1\leq d\leq n-1$.  Then we have 
$$(I:x)=(x_1,\ldots,x_{n-1})^{(d-1)}, \ \ \text{and} \ \ I+(x)=(x_1,\ldots,x_{n-1})^{(d)}+(x_n)$$
which are both radical by the induction hypothesis.
Also note that $(I:x^2)=(I:x)$ since $I$ is generated by square-free monomials.  It therefore follows from Lemma \ref{lem:PRS} that the ideal $I=(x_1,\ldots,x_n)^{(d)}$ is radical.
\end{proof}
Applying principal radical systems to Specht ideals requires the following decomposition of shifted Specht ideals, and is key to the further results of this paper.  This is Theorem \ref{thm:B1} from the Introduction.
\begin{theorem}
	\label{thm:radD}
	For any integers $k,d$ satisfying $0\leq k<d\leq n-k$, we have 
	\begin{equation}
	\label{eq:radD}
	\mathfrak{a}(n,k,d)=\mathfrak{a}(n,k,d-1)\cap (x_1,\ldots,x_n)^{(d)}.
	\end{equation}
\end{theorem}
Before embarking on the proof of Theorem \ref{thm:radD}, we will show how to use Theorem \ref{thm:radD} and principal radical systems to show that shifted Specht ideals are radical.  This is Theorem \ref{thm:B} from the Introduction.
\begin{theorem}
	\label{thm:rad}
For any integers $k,d$ satisfying $0\leq k\leq d\leq n-k$, the shifted Specht ideal $\mathfrak{a}(n,k,d)$ is radical.	
\end{theorem}
\subsection{Proof of Theorem \ref{thm:B} or Theorem \ref{thm:rad}}
\begin{proof}
	Assuming that Theorem \ref{thm:radD} holds, we prove Theorem \ref{thm:rad} by induction on $n\geq 2$.  For the base case $n=2$, the only possibilities for integers $k,d$ are $k=0$ and $d=1,2$ and $k=1=d$.  In the case $k=0$ we have $\mathfrak{a}(2,0,d)=(x_1,x_2)^{(d)}$ which is radical by Lemma \ref{lem:monomialrad}.  In the case, $k=d=1$ we have  $\mathfrak{a}(2,1,1)=((x_1-x_2))$ is prime, and therefore radical.  For the inductive step, assume that $\mathfrak{a}(n-1,j,e)$ for all integers satisfying $0\leq j\leq e\leq n-1-j$.  Fix integers $k,d$ satisfying $1\leq k= d\leq n-k$.  First we argue for the case $d=k$.  Let $I=\mathfrak{a}(n,k,k)$ and $x=x_n$.  Then by Corollary \ref{cor:CoC} we have 
	$$I+(x)=\mathfrak{a}(n,k,k)+(x_n)=\mathfrak{a}(n-1,k-1,k)+(x_n).$$
	By Theorem \ref{thm:radD} we have 
	$$\mathfrak{a}(n-1,k-1,k)=\mathfrak{a}(n-1,k-1,k-1)\cap (x_1,\ldots,x_{n-1})^{(k)}.$$
	By the induction hypothesis $\mathfrak{a}(n-1,k-1,k-1)$ is radical, and by Lemma \ref{lem:PRSmonomial} $(x_1,\ldots,x_{n-1})^{(k)}$ is radical.  It follows that $\mathfrak{a}(n-1,k-1,k)$ and hence also $I+(x)$ is also radical.
	
	Also using the change of coordinates map $\Phi\colon x_i\mapsto y_i$ in Corollary \ref{cor:CoC} we find that 
	\begin{align*}
	\Phi\left(I:x\right)=\left(\Phi\left(I\right):\Phi(x_n)\right)=(\mathfrak{a}(n-1,k-1,k):x_n)=\mathfrak{a}(n-1,k-1,k)=\Phi\left(I\right)
	\end{align*}
	from which it follows that $(I:x)=I$.  Therefore it follows from Lemma \ref{lem:PRS} that $I=\mathfrak{a}(n,k,k)$ itself must be radical.
	
	For $d>k$ we appeal again to Theorem \ref{thm:radD}: 
	$$\mathfrak{a}(n,k,d)=\mathfrak{a}(n,k,d-1)\cap (x_1,\ldots,x_n)^{(d)}=\mathfrak{a}(n,k,k)\cap (x_1,\ldots,x_n)^{(d)}.$$
	Since $\mathfrak{a}(n,k,k)$ is radical, and $(x_1,\ldots,x_n)^{(d)}$ is radical, it follows that $\mathfrak{a}(n,k,d)$ is radical too.
\end{proof}

\subsection{Proof of Theorem \ref{thm:B1} or Theorem \ref{thm:radD}}
The proof of Theorem \ref{thm:radD} (or Theorem \ref{thm:B1} from the Introduction) comes in three steps, each of which we state as a lemma.  First some notation:  For any exponent vector $\a=(a_1,\ldots,a_n)\in\N^n$ we denote the associated monomial by $\x^\a=x_1^{a_1}\cdots x_n^{a_n}$, and its radical by $\sqrt{\x^\a}=\prod_{a_i>0}x_i$.  
\begin{lemma}
	\label{lem:monomialrad}
	The ideal $\mathfrak{a}(n,k,d-1)\cap (x_1,\ldots,x_n)^{(d)}$ is generated by products of monomials and polynomials in the shifted Specht module $V(n,k,d-1)$.  In fact if the sum of any monomials times forms in $V(n,k,d-1)$ lies in the intersection $\mathfrak{a}(n,k,d-1)\cap (x_1,\ldots,x_n)^{(d)}$, then so do each of its summands.
\end{lemma}
\begin{proof}
	It is not difficult to see that every polynomial $P\in\mathfrak{a}(n,k,d-1)$ decomposes into a sum of terms of the form
	$$P=\sum_{\a\in\N^{n}}\x^\a\cdot \ \nu_\a$$
	where $\nu_\a\in V(n,k,d-1)$.  We want to show that if $P\in (x_1,\ldots,x_n)^{(d)}$, then each of its summands are too, i.e. $\x^\a\cdot\nu_\a\in (x_1,\ldots,x_n)^{(d)}$ for all $\a\in\N^n$.  Suppose by way of contradiction that for some $\a\in\N^n$ and some $\nu_\a\in V(n,k,d-1)$ that $\x^\a\cdot \nu_\a\notin (x_1,\ldots,x_n)^{(d)}$.  Since $(x_1,\ldots,x_n)^{(d)}$ is a monomial ideal, there must be some monomial $\x^\b$ which appears in the monomial expansion of $\x^\a\cdot \nu_\a$ such that $\x^\b\notin (x_1,\ldots,x_n)^{(d)}$.  Define the \emph{weight} of monomial $\x^\b$ as $\operatorname{wt}(\x^\b)=\#\{b_i>0\}$.  Since $(x_1,\ldots,x_n)^{(d)}$ consists of all monomials of weight at least $d$, it follows that $\operatorname{wt}(\x^\b)\leq d-1$.  On the other hand, since $\nu_\a$ is a linear combination of shifted Specht polynomials of type $\lambda(d)$, it follows that every monomial in the monomial expansion of $\nu_\a$ also has weight $d-1$.  This implies that $\operatorname{wt}(\x^\b)=d-1$, and therefore that 
	$$\frac{\x^\b}{\sqrt{\x^\b}} = \x^\a.$$
	In particular, we see that the monomial $\x^\b$ is unique to the term $\x^\a\cdot\nu_\a$, and hence must occur with the same coefficient in the monomial expansion of $\x^\a\cdot\nu_\a$ as it does in the monomial expansion of $P=\sum_{\a\in\N^n}\x^\a\cdot\nu_\a$.  Therefore $P\notin (x_1,\ldots,x_n)^{(d)}$, as desired.      
\end{proof}
Lemma \ref{lem:monomialrad} tells us that it suffices to check equation \eqref{eq:radD} in Theorem \ref{thm:radD} on products of monomials with $V(n,k,d-1)$.  So we want to show that for each  $\nu\in V(n,k,d-1)$ and for each $\a\in\N^n$ the following implication holds:
$$\x^\a\cdot \nu\in (x_1,\ldots,x_n)^{(d)} \ \Rightarrow \ \x^\a\cdot \nu\in \mathfrak{a}(n,k,d).$$
Note that since $(x_1,\ldots,x_n)^{(d)}$ is generated by square-free monomials, we have 
$$\x^\a\cdot \nu\in (x_1,\ldots,x_n)^{(d)} \ \Leftrightarrow \ \sqrt{\x^\a}\cdot \nu\in (x_1,\ldots,x_n)^{(d)}.$$
In particular, we may assume without loss of generality that our monomials $\x^\a$ are square-free.  For any polynomial $F\in R$ define its \emph{support} to be the set of square-free monomials which divide some non-zero monomial term of $F$.  For example, given a tableau $T\in\tab(n,k,d)$, the support of the shifted Specht polynomial $F_T(d-1)$ is the set of square-free monomials indexed by subsets of numbers in the support of $T$, no two of which lie in the same column of $T$.  
 
\begin{lemma}
	\label{lem:suppNo}
		For each $T\in\tab(n,k,d)$, if $\x^\a\notin\supp(F_T(d-1))$ then $\x^\a\cdot F_T(d-1)\in \mathfrak{a}(n,k,d)$.
\end{lemma}
\begin{proof}
	If there is a variable $x_i\in \supp(\x^\a)$ which is not in $\supp(F_T(d-1))$, then certainly $\x^\a\cdot F_T(d-1)\in \mathfrak{a}(n,k,d)$.  Otherwise, there must be two indices $i\neq j$ such that $x_i,x_j\in \supp(\x^\a)$ and $i,j$ in the same column of $T$.  Choose any index $r\neq i,j$ such that $x_r\notin\supp(F_T(d-1))$, which exists since $0\leq k\leq d-1<d\leq n-k$, and let $(i,r), (j,r)\in\mathfrak{S}_n$ be the transpositions swapping $i,r$ and $j,r$, respectively.  Then we have 
	$$F_T(d-1)=F_{(i,r).T}(d-1)+F_{(j,r).T}(d-1)$$
	and since $x_i\cdot F_{(i,r).T}(d-1), x_j\cdot F_{(j,r).T}(d-1)\in \mathfrak{a}(n,k,d)$, it follows that 
	$$\x^\a\cdot F_T(d-1)\in\mathfrak{a}(n,k,d)$$
	as desired.
\end{proof}

Finally we must show what happens with square-free monomials which do lie in the support of the shifted Specht polynomials.  Here we use symmetry to make the further reduction that our square-free monomial is initial, i.e. 
$$\x^\a=\x^\m=x_1\cdots x_m, \ \ \text{for some} \ 1\leq m\leq d-1.$$
Setting $V^m(n,k,d-1)\subset V(n,k,d-1)$ to be the span of shifted Specht polynomials indexed by standard tableaux $T\in\stab(n,k,d-1)$ for which $\{1,\ldots,m\}\not\subset\operatorname{supp}(T)$.  Then the shifted Specht module decomposes into a direct sum
$$V(n,k,d-1)=V_m(n,k,d-1)\oplus V^m(n,k,d-1),$$
and Lemma \ref{lem:suppNo} says $\x^\m\cdot \nu\in \mathfrak{a}(n,k,d)$ for every $\nu\in V^m(n,k,d-1)$.  
Recall the bijective linear map 
$$\xymatrixrowsep{.5pc}\xymatrix{\phi\colon V(n,k,d-1)\ar[r] & V([n]_m,k-m,d-m)\\
F_T(d)\ar@{|->}[r] & F_{T'}(d)}$$
where
\begin{equation}
\label{eq:Trad}
T=\begin{ytableau}
\none & \none & \none & 1 & \cdots & m & i_{m+1} & \cdots & i_{k} & i_{k+1} & \cdots & i_{d-1}\\
i_{d} & \cdots & i_{n-k} & j_1 & \cdots & j_m & j_{m+1} & \cdots & j_k\\
\end{ytableau}
\end{equation}	
and 
\begin{equation}
\label{eq:T'rad}
T'=\begin{ytableau}
\none & \none & \none & \none & \none & \none & i_{m+1} & \cdots & i_{k} & i_{k+1} & \cdots & i_{d-1}\\
i_{d} & \cdots & i_{n-k} & j_1 & \cdots & j_m & j_{m+1} & \cdots & j_k\\
\end{ytableau}
\end{equation}

\begin{lemma}
\label{lem:suppYes}
Fix $\nu\in V_m(n,k,d-1)$.  If $\x^\m\cdot \nu\in (x_1,\ldots,x_n)^{(d)}$ then $\nu=0$, and hence $\x^\m\cdot \nu\in \mathfrak{a}(n,k,d)$.	
\end{lemma}
\begin{proof}
	We observe that 
	$$\x^\m\cdot \left(\nu-\x^\m\cdot \nu'\right)\in (x_1,\ldots,x_n)^{(d)}$$
	where $\nu'\in V([n]_m,k-m,d-1-m)$ is the image of $\nu$ in the map above.  Indeed note that for each standard tableau $T\in\stab_m(n,k,d-1)$ as in \eqref{eq:Trad}, the support of the difference $F_T(d-1)-\x^\m\cdot F_{T'}(d-1-m)$ does not contain the monomial $\x^\m$, hence the product $\x^\m\cdot\left(F_T(d)-\x^\m\cdot F_{T'}(d-1-m)\right)\in (x_1,\ldots,x_n)^{(d)}$.  Hence if $\x^\m\cdot \nu\in (x_1,\ldots,x_n)^{(d)}$, it follows that $\left(\x^\m\right)^2\cdot \nu'\in (x_1,\ldots,x_n)^{(d)}$, and hence that $\nu'\in \left((x_1,\ldots,x_n)^{(d)}:\x^\m\right)=(x_{m+1},\ldots,x_n)^{(d-m)}$.  For degree reasons this implies that $\nu'=0$, and hence by Corollary \ref{cor:inductive}, $\nu=0$ as well. 
\end{proof}

\begin{proof}[Proof of Theorem \ref{thm:radD}]
	The containment $\mathfrak{a}(n,k,d)\subset \mathfrak{a}(n,k,d-1)\cap (x_1,\ldots,x_n)^{(d)}$ is clear.  For the reverse containment, Lemma \ref{lem:monomialrad} implies that it suffices to check it on products of monomials with polynomials in $V(n,k,d-1)$.  Since $(x_1,\ldots,x_n)^{(d)}$ is generated by square-free monomials we may assume that our monomials are square-free, and by symmetry we may assume that our monomial has the form $\x^\m=x_1\cdots x_m$ for some integer $1\leq m\leq d$.  Then as above we have 
	$$V(n,k,d-1)=V_m(n,k,d-1)\oplus V^m(n,k,d-1)$$
	and by Lemma \ref{lem:suppYes}, $\x^\m\cdot \nu\in \mathfrak{a}(n,k,d)$ if $\nu\in V_m(n,k,d-1)$.  Furthermore, Lemma \ref{lem:suppNo} implies that for $\nu\in V^m(n,k,d-1)$ if $\x^\m\cdot \nu\in (x_1,\ldots,x_n)^{(d)}$ then $\x^\m\cdot \nu\in \mathfrak{a}(n,k,d)$, and the result follows.
\end{proof}

\section{Perfection of Specht and Specht-Monomial Ideals}
\label{sec:perfection}
Recall that an ideal $I\subset R$ in a commutative ring has projective or homological dimension $s$ if a minimal resolution of $R/I$ as an $R$-module has length $s$.  Its grade is the length $g$ of a maximal $R$-sequence contained in $I$, or equivalently the smallest integer $g$ for which the Ext group $\operatorname{Ext}_R^g(R/I,R)$ is non-zero.  We say that the ideal $I\subset R$ is \emph{perfect} if its projective dimension is equal to its grade.  In our case, where $R$ is polynomial (hence Cohen-Macaulay), the grade of an ideal is equal to its height, and by the Auslander-Buchsbaum formula, $I$ is perfect if and only if the quotient $R/I$ is Cohen-Macaulay.  For more details on these matters we refer the reader to \cite{Matsumura}. 

As in Section \ref{sec:radical}, the main idea here is to use principal radical systems.
\begin{lemma}
	\label{lem:perfectPRS}
	Let $I\subset R$ be any homogeneous ideal and let $x\in R\setminus I$ be any homogeneous polynomial.  Then 
	\begin{enumerate}
		\item if $(I:x)=I$ and if $I+(x)$ is perfect, then $I$ is perfect too.
		\item if $(I:x)\neq I$ and if $(I:x)$ and $I+(x)$ are both perfect of the same grade $g$, then $I$ is perfect of that same grade too.
	\end{enumerate}
\end{lemma}
\begin{proof}
	Item (1.) is well known, and can be found in any commutative algebra text, e.g. \cite[Theorem 17.3]{Matsumura}.  For (2.) we use the long exact sequence for $\operatorname{Ext}$-modules associated with the short exact sequence of $R$-modules
	$$\xymatrix{0\ar[r] & R/(I:x)\ar[r]^-{\times x} & R/I\ar[r] & R/I+(x)\ar[r] & 0.}$$
\end{proof}
  
As in Section \ref{sec:radical}, we give an easy application of principal radical systems.
\begin{lemma}
	\label{lem:PRSperfect}
	For every integer $d$ satisfying $1\leq d\leq n$ the monomial ideal 
	$$(x_1,\ldots,x_n)^{(d)}$$
	is perfect.
\end{lemma}
\begin{proof}
	By induction on $n$, where the base case $n=1$ is trivial.  For the induction step, assume that $(x_1,\ldots,x_{n-1})^{(e)}$ is perfect for every $1\leq e\leq n-1$, and fix an integer $d$ satisfying $1\leq d\leq n$.  Set $I=(x_1,\ldots,x_n)^{(d)}$ and $x=x_n$.  Then we have 
	\begin{align*}
	(I:x)= & (x_1,\ldots,x_{n-1})^{(d-1)} & I+(x)=(x_1,\ldots,x_{n-1})^{(d)}+(x_n) 
	\end{align*}
	which are both perfect by our induction hypothesis.  It follows from Lemma \ref{lem:perfectPRS} that $I=(x_1,\ldots,x_n)^{(d)}$ is perfect too.\footnote{Alternatively, one could also appeal to Reisner's theorem \cite{Reisner} here, since $(x_1,\ldots,x_n)^{(d)}$ is the Stanley-Reisner ideal of the $d-1$ skeleton of an $n-1$ simplex.}
\end{proof}

It is trickier to apply Lemma \ref{lem:perfectPRS} to shifted Specht ideals.  For one thing, not all shifted Specht ideals are perfect.  Indeed,  
Theorem \ref{thm:radD} gives the decomposition 
$$\mathfrak{a}(n,k,d)=\mathfrak{a}(n,k,k)\cap (x_1,\ldots,x_n)^{(d)}.$$
Individually the grades (=heights) of the ideals $\mathfrak{a}(n,k,k)$ and $(x_1,\ldots,x_n)^{(d)}$ are $g=n-k$ and $g=n-d+1$, respectively.
In particular we see that if $d>k+1$, then the two ideals $\mathfrak{a}(n,k,k)$ and $(x_1,\ldots,x_n)^{(d)}$ have mixed heights, and in particular the shifted Specht ideal $\mathfrak{a}(n,k,d)$ cannot be perfect.  For $d=k+1$, we must show that the intersection is perfect:
$$\mathfrak{a}(n,k,k+1)=\mathfrak{a}(n,k,k)\cap (x_1,\ldots,x_n)^{(d)}.$$
To this end, we appeal to the following basic fact, which appears in the paper \cite{HE} by Hochster-Eagon, and we refer the reader there for its proof.
  
\begin{lemma}[Proposition 18, \cite{HE}]
	\label{lem:HE}
	Suppose that $I,J\subset R$ are two perfect ideals of the same grade $g$, and assume that $I+J$ has grade $g+1$.  Then $I+J$ is perfect if and only if $I\cap J$ is perfect. 
\end{lemma}

In the spirit of Lemma \ref{lem:HE} we study the following sum of ideals, which plays a key role in this paper.
	\begin{definition}
		\label{def:INK}
Fix an integer $k$ satisfying $1\leq k<k+1\leq n-k$, and define the \emph{Specht-monomial ideal} for the pair $(k,n)$ to be the sum of ideals
\begin{equation}
\label{eq:INK}
I(n,k)=\mathfrak{a}(n,k,k)+(x_1,\ldots,x_n)^{(k+1)}.
\end{equation}
	\end{definition} 

The next result is key to our results in this section, and it plays a similar role as Theorem \ref{thm:radD} in Section \ref{sec:radical}.  Note however that this result depends on the characteristic of $\F$.  It is Theorem \ref{thm:C}(3.) from the Introduction.
\begin{theorem}
	\label{thm:perfectD}
	Let $p=\operatorname{char}(\F)\geq 0$ and fix positive integers $n,k$ satisfying $n\geq 2k+1$.  Then the following conditions are equivalent.
	\begin{enumerate}
		\item $p=0$ or $p\geq k+1$.
		\item The Specht monomial ideal $I(n,k)$ satisfies
		\begin{align}
		\label{eq:INKD}
		I(n,k)= & I(n-1,k-1)\cap \left(\ydeal\right)\\
		\nonumber\text{where} \ \ \ & y_i=\begin{cases} x_n-x_i & \text{if} \ 1\leq i\leq n-1\\ x_n & \text{if} \ i=n\\ \end{cases}.
		\end{align}	
	\end{enumerate} 
\end{theorem}

Before embarking on the proof of Theorem \ref{thm:perfectD}, we will use it in conjunction with principal radical systems to prove Theorems \ref{thm:C}(4.) and also Theorem \ref{thm:SpechtP} from the Introduction.

\begin{theorem}
	\label{thm:perfect}
	Let $p=\operatorname{char}(\F)\geq 0$ and fix positive integers $n,k$ satisfying $n\geq 2k+1$.  Then the following conditions are equivalent.
	\begin{enumerate}
		\item $p=0$ or $p\geq k+1$.
		\item The Specht-monomial ideal $I(n,k)$ is perfect.
	\end{enumerate}
\end{theorem}

\begin{theorem}
	\label{thm:perfectSpecht}
	Let $p=\operatorname{char}(\F)\geq 0$ and fix positive integers $n,k$ satisfying $n\geq 2k+1$.  
	\begin{enumerate}
		\item If $p=0$ or $p\geq k+1$, then the Specht ideal $\mathfrak{a}(n+1,k+1,k+1)$ is perfect.
		\item If $n\geq 2p+1$, then the Specht ideal $\mathfrak{a}(n+1,p+1,p+1)$ is not perfect.
	\end{enumerate}
\end{theorem}

The proof for Theorem \ref{thm:perfect} occupies the next subsection, followed by the proof of Theorem \ref{thm:perfectSpecht}, and we save the proof of Theorem \ref{thm:perfectD} for the end.

\subsection{Proof of Theorem \ref{thm:C}(4.) or Theorem \ref{thm:perfect}}
We will focus on the implication (1.) $\Rightarrow$ (2.) first. Assuming that $p=0$ or $p\geq k+1$, Theorem \ref{thm:perfectD} says that for all $1\leq k<k+1\leq n-k$ the Specht-monomial ideal satisfies \eqref{eq:INKD}:
$$I(n,k)=I(n-1,k-1)\cap \left(\ydeal\right).$$

We need the following Lemma, which tells us how to go between $x$-variables and $y$-variables, and provides a direct link to Lefschetz properties from Section \ref{sec:Lefschetz}.
\begin{lemma}
	\label{lem:Dlem}
	Let $P^j(x_1,\ldots,x_{n-1})$ be any square-free polynomial of degree $j$ in the variables $x_1,\ldots,x_{n-1}$.  Then modulo the principal ideal $(x_n^2)$ we have 
	\begin{align*}
	P^j(x_1,\ldots,x_{n-1})\equiv & (-1)^j\left(P^j(y_1,\ldots,y_{n-1})-x_n\cdot D\left(P^j(y_1,\ldots,y_{n-1}\right)\right) & \text{mod} \ (x_n^2)
	\end{align*} where $D$ is the linear partial differentiation operator ${\dsp D=\frac{\partial}{\partial y_1}+\cdots+\frac{\partial}{\partial y_{n-1}}}$.  In particular, if $\alpha=\alpha(x_1,\ldots,x_{n-1})\in V(n-1,k-1,k-1)$ is a linear combination of Specht polynomials then 
	$$\alpha(x_1,\ldots,x_{n-1})=(-1)^{k-1}\cdot \alpha(y_1,\ldots,y_{n-1}).$$
\end{lemma}
\begin{proof}
	By linearity of $D$ it suffices to assume that $P^j$ is a square-free monomial, and by symmetry we may assume is $P^j(x_1,\ldots,x_{n-1})=x_1\cdots x_j$.  Then we have 
	\begin{align*}
	P^j(x_1,\ldots,x_{n-1})= & x_1\cdots x_j= (x_n-y_1)\cdots(x_n-y_j)\\
	= & x_n^2\cdot \left(\text{stuff}\right)+(-1)^{j-1}\cdot x_n\left(\sum_{i=1}^j y_1\cdots \hat{y_i}\cdots y_j\right)+(-1)^j\cdot y_1\cdots y_j\\
	 & (\text{where $\hat{y}$ means omission})\\
	 = & (-1)^j\left(y_1\cdots y_j - x_nD\left(y_1\cdots y_j\right)\right)+x_n^2\cdot \left(\text{stuff}\right)\\
	 \equiv & (-1)^j\left(P^j(y_1,\ldots,y_{n-1})-x_nD\left(P^j(y_1,\ldots,y_{n-1})\right)\right) & \text{mod} \ (x_n^2)
	\end{align*}
	as claimed.
	The second statement follows from the first since $D(\alpha)=0$.
\end{proof}

Next, for each pair of positive integers $n,k$ satisfying $n\geq 2k+1$ let us form the new ideal
\begin{equation}
\label{eq:JNK}
J(n,k)=I(n-1,k-1)+\ydeal.
\end{equation}

\begin{lemma}
	\label{lem:JNK}
	Assume that $p=0$ or $p\geq k+1$.  Then the ideal $J(n,k)$ satisfies
	$$J(n,k)+(x_n)=I(n-1,k-1)+(x_n), \ \ \text{and} \ \ (J(n,k):x_n)=(x_1,\ldots,x_{n-1})^{(k-1)}+(x_n).$$
\end{lemma}
\begin{proof}
	We have 
	\begin{align*}
	J(n,k)+(x_n)= & I(n-1,k-1)+\ydeal+(x_n)\\
	= & \mathfrak{a}(n-1,k-1,k-1)+(x_1,\ldots,x_{n-1})^{(k)}+(y_1,\ldots,y_{n-1})^{(k)}+(x_n)\\
	= & \mathfrak{a}(n-1,k-1,k-1)+(x_1,\ldots,x_{n-1})^{(k)}+(x_n)\\
	= & I(n-1,k-1)+(x_n).
	\end{align*}
	For the other equality, note first that containment $(J(n,k):x_n)\supseteq (x_1,\ldots,x_{n-1})^{(k-1)}+(x_n)=(y_1,\ldots,y_{n-1})^{(k-1)}+(x_n)$ follows from Lemma \ref{lem:Dlem}.  Indeed identifying the space of square-free polynomials with the monomial complete intersection $B=\F[x_1,\ldots,x_{n-1}]/(y_1^2,\ldots,y_{n-1}^2)$, Lemma \ref{lem:Vnkk} implies that the derivative map $D\colon B_{k}\rightarrow B_{k-1}$ is surjective, and hence  for any square-free monomial of degree $k-1$ in variables $y_1,\ldots,y_{n-1}$, say $P=y_1\cdots y_{k-1}$, we know by Lemma \ref{lem:Vnkk} there is square-free polynomial of degree $k$ for which
	$$D(Q(y_1,\ldots,y_{n-1}))=P(y_1,\ldots,y_{n-1})=y_1\cdots y_{k-1}.$$
	Then Lemma \ref{lem:Dlem} implies that  
	\begin{align*}
	P\cdot x_n= & y_1\cdots y_{k-1}\cdot x_n=x_n\cdot D\left(Q(y_1,\ldots,y_{n-1})\right)\\
	\equiv & Q(x_1,\ldots,x_{n-1})\pm Q(y_1,\ldots,y_{n-1}) & \text{mod} \ (x_n^2)	\end{align*}  
	Since $Q(x_1,\ldots,x_{n-1})\in I(n-1,k-1)$ and $Q(y_1,\ldots,y_{n-1})\in (y_1,\ldots,y_{n-1})^{(k)}$, it follows that $x_n\cdot P=x_n\cdot D(Q(y_1,\ldots,y_{n-1}))\in J(n,k)$, and hence $P=y_1\cdots y_{k-1}\in (J(n,k):x_n)$.  
	For the reverse containment, fix $G\in (J(n,k):x_n)$.  Then for each $S\in\stab(n-1,0,k)$ there exists polynomials $d_S\in R$ for which 
	$$x_n G-\sum_{S\in \stab(n-1,0,k)}d_SM_S\in I(n-1,k-1)+(x_n^2)$$
	where $M_S=M_S(y_1,\ldots,y_{n-1})\in (y_1,\ldots,y_{n-1})^{(k)}$ are square-free monomials of degree $k$ in the $y$-variables.  By Lemma \ref{lem:Dlem} we have for each $S\in\stab(n-1,0,k)$ 
	\begin{align*}
	M_S(y_1,\ldots,y_{n-1})\equiv & \pm \left(M_S(x_1,\ldots,x_{n-1})-x_n D(M_S(x_1,\ldots,x_{n-1}))\right) & \text{mod} \ (x_n^2)
	\end{align*}
	and since $M_S(x_1,\ldots,x_{n-1})\in (x_1,\ldots,x_{n-1})^{(k)}\subset I(n-1,k-1)$, we see that 
	\begin{align*}
	x_nG-\sum_{S\in \stab(n-1,0,k)}d_SM_S(y_1,\ldots,y_{n-1})\equiv & x_nG- \sum_{S\in \stab(n-1,0,k)}d_Sx_nD(M_S(y_1,\ldots,y_{n-1}))\\
	\equiv & 0 \ \  \text{mod} \ I(n-1,k-1)+(x_n^2)
	\end{align*}
	Since $x_n$ is a non-zero divisor for $I(n-1,k-1)$, it follows that 
	$$G-\sum_{S\in \stab(n-1,0,k)}d_SD(M_S(x_1,\ldots,x_{n-1}))\in I(n-1,k-1)+(x_n)\subset (x_1,\ldots,x_{n-1})^{(k-1)}+(x_n)$$
	and since $D(M_S(x_1,\ldots,x_{n-1}))\in (x_1,\ldots,x_{n-1})^{(k-1)}$ for all $S\in\stab(n-1,0,k)$, we see also that $G\in (x_1,\ldots,x_{n-1})^{(k-1)}+(x_n)$, as desired.
\end{proof}

We are now in a position to prove implication (1.) $\Rightarrow$ (2.) in Theorem \ref{thm:perfect}.
\begin{proof}[Proof of (1.) $\Rightarrow$ (2.) in Theorem \ref{thm:perfect}]
	Assume that $p=0$ or $p\geq k+1$.  We prove by induction on $n\geq 3$ for each integer $k$ satisfying $1\leq k<k+1\leq n-k$ then the ideal $I(n,k)$ is perfect.  The base case is $n=3$ where the only possible $k$ value is $k=1$.  Here the assumption on $p$ is vacuous, and we have  
	\begin{align*}
	I(3,1)= & \mathfrak{a}(3,1,1)+(x_1,x_2,x_3)^{(2)}=(x_1-x_3,x_2-x_3,x_1x_2,x_1x_3,x_2x_3).
	\end{align*}
	Note that $\F[x_1,x_2,x_3]/I(3,1)\cong \F[z]/(z^2)$ is Cohen-Macaulay, which implies that $I(3,1)$ is perfect of grade $n-k+1=3$.  
	
	For the inductive step, assume that $I(n-1,j)$ is perfect for all integers $j$ satisfying $1\leq j<j+1\leq n-1-j$.  Fix $k$ satisfying $1\leq k<k+1\leq n-k$, and consider the Specht-monomial ideal
	\begin{align*}
	I(n,k)= & \mathfrak{a}(n,k,k)+(x_1,\ldots,x_n)^{(k+1)}.
	\end{align*}
	Consider the sum $J(n,k)$ as in \eqref{eq:JNK}, i.e.
	$$J(n,k)=I(n-1,k-1)+(y_1,\ldots,y_{n-1})^{(k)}+(x_n^2).$$  
	Note that $I(n-1,k-1)$ is perfect by the induction hypothesis.  Also $(y_1,\ldots,y_{n-1})^{(k)}+(x_n^2)$ is perfect since $(y_1,\ldots,y_{n-1})^{(k)}$ is perfect (by Lemma \ref{lem:PRSperfect}), and $x_n^2$ is $(y_1,\ldots,y_{n-1})^{(k)}$-regular.  Therefore by Lemma \ref{lem:HE}, $I(n,k)$ is perfect if and only if $J(n,k)$ is perfect.  But according to Lemma \ref{lem:JNK} $J(n,k)+(x_n)=I(n-1,k-1)+(x_n)$ and $(J(n,k):x_n)=(x_1,\ldots,x_{n-1})^{(k-1)}+(x_n)$ which are both perfect of the same grade $g=n-k+2$, and hence by Lemma \ref{lem:PRSperfect}, it follows that $J(n,k)$ is also perfect (of grade $g=n-k+2$).  Thus it follows that $I(n,k)$ is perfect completing the induction step, and hence the proof.
\end{proof}

Next we prove the reverse implication (2.) $\Rightarrow$ (1.) in Theorem \ref{thm:perfect} which also requires a bit of a set up.  First note that Theorem \ref{thm:perfectD} implies that if $p=0$ or $p\geq k+1$, then $I(n,k)$ has the following irredundant primary decomposition:
\begin{equation}
\label{eq:InkPD}
I(n,k)=\mathfrak{a}(n,k,k)+(x_1,\ldots,x_n)^{(k+1)}=\bigcap_{\sigma\in \mathfrak{S}_n}\underbrace{\sigma.\left(x_1-x_2,\ldots,x_1-x_{n-k+1},x_1^2\right)}_{Q_\sigma}
\end{equation}
Note that for $\sigma=e$, we have $Q_e=\mathfrak{a}(n-k+1,1,1)+(x_1,\ldots,x_{n-k+1})^{(2)}$; in particular, $Q_e$ contains all quadratic forms in the variables $x_1,\ldots,x_{n-k+1}$.
\begin{lemma}
	\label{lem:Nagata}
	Over any field $\F$, and for any positive integers $n,k$ satisfying $n\geq 2k+1$, if $I(n,k)$ is perfect then $I(n,k)$ must satisfy the decomposition \eqref{eq:InkPD}.
\end{lemma}
\begin{proof}
	Assume that $I(n,k)$ is perfect, and consider the intersection of primary ideals as in \eqref{eq:InkPD}, i.e.
	$$I'(n,k)=\bigcap_{\sigma\in \mathfrak{S}_n}Q_{\sigma}$$
	with minimal associated prime divisors given by 
	$$P_\sigma=\sqrt{Q_{\sigma}}=\sigma.\left(x_1,\ldots,x_{n-k+1}\right).$$
	We would like to show that $I(n,k)=I'(n,k)$.  Note first that they have the same radical:
	\begin{equation}
	\label{eq:INKRad}
	\sqrt{I'(n,k)}=\bigcap_{\sigma\in\mathfrak{S}_n}P_{\sigma}=(x_1,\ldots,x_n)^{(k)}=\sqrt{I(n,k)}.
	\end{equation}
	For the last equality, the containment $I(n,k)\subseteq (x_1,\ldots,x_n)^{(k)}$ is clear, and the other containment follows from the containment: 
	$$x_{i_1}^2\cdots x_{i_k}^2=x_{i_1}\cdots x_{i_k}\cdot F_T+\left(\text{stuff in} \ (x_1,\ldots,x_n)^{(k+1)}\right)\in I(n,k)$$
	where 
	$$T=\begin{ytableau}
	\none & \none & \none & i_1 & \cdots & i_k\\
	i_{k+1} & \cdots & i_{n-k} & j_1 & \cdots & j_k\\
	\end{ytableau}$$
	It follows that $\{P_\sigma \ | \ \sigma\in\mathfrak{S}_n\}$ is a complete list of minimal prime divisors of $I(n,k)$, which implies equation \eqref{eq:INKRad}.
	
	Set $I'=I'(n,k)$ and $I=I(n,k)$.  Suppose that a primary decomposition of $I$ is given by 
	$I=U_1\cap\cdots\cap U_m$.  Since $I$ is perfect, all of its associated prime divisors must be minimal, and hence the primary components must be indexed by the symmetric group and we can write 
	$$I=\bigcap_{\sigma\in \mathfrak{S}_n}U_{\sigma}$$
	where $\sqrt{U_\sigma}=P_{\sigma}$.
	Fix $\sigma\in\mathfrak{S}_n$, and set $U=U_\sigma$, $P=P_\sigma$, and $Q=Q_\sigma$.  We want to show that $U_\sigma=Q_\sigma$.  Let $R_P$ be the polynomial ring $R=\F[x_1,\ldots,x_n]$ localized at the prime ideal $P$.  By a theorem of Nagata \cite[Theorem 8.7]{Nagata} we have 
	$$U=IR_P\cap R \ \ \text{and} \ \ Q=JR_P\cap R.$$
	We will prove that in the local ring $R_P$, the ideals are equal $IR_P=JR_P$.  Certainly because of the containment $I\subseteq J$, we have also $IR_P\subseteq JR_P$.  In the other direction we observe that $JR_P=QR_P$.  Without loss of generality we may assume that $\sigma=e$, and $Q=(x_1-x_2,\ldots,x_{1}-x_{n-k+1},x_1^2)$.  For each pair $1\leq r<s\leq n-k+1$ we may choose $1\leq i_1<\cdots<i_{k-1}\leq n-k+1$ such that $r,s\notin\left\{i_1,\ldots,i_{k-1}\right\}$ and setting $j_i=n-k+1+i$ define the tableau
	$$T=\begin{ytableau}
	\none & \none & \none & i_1 & \cdots & i_{k-1} & r\\
	i_{k+1} & \cdots & i_{n-k} & j_1 & \cdots & j_{k-1} & s\\
	\end{ytableau}$$
	Then in the local ring $R_P$ the Specht polynomial $F_T\in I$ has the form $F_T=u\cdot (x_r-x_s)$ where $u\in R_P$ is a unit.  Also consider the monomial
	$$M_S=x_{n-k+2}\cdots x_n\cdot x_r\cdot x_s\in (x_1,\ldots,x_n)^{(k)}\subset I.$$
	Then in the local ring $R_P$ it has the form $M=w\cdot x_r\cdot x_s$ where $w\in R_P$ is a unit.  Therefore the generators of $QR_P$ satisfy $(x_r-x_s)\in IR_P$ and $x_r\cdot x_s\in IR_P$ it follows that $QR_P\subset IR_P$ and hence that $JR_P\subset IR_P$, as desired.
\end{proof}

In particular Lemma \ref{lem:Nagata} implies that if $I(n,k)$ is not perfect then it must have an embedded prime divisor.  The upshot here is that embedded prime divisors seem to be an easier check for imperfection than, say, extraneous syzygies in a minimal free resolution.  Note that Theorem \ref{thm:rad} implies that the Specht ideals $\mathfrak{a}(n+1,k+1,k+1)$ never have embedded prime divisors in any characteristic, and this is one of the main reasons for the disparity between Theorems \ref{thm:perfect} and \ref{thm:perfectSpecht}.  We are now in a position to prove the remainder of Theorem \ref{thm:perfect}.

\begin{proof}[Proof of (2.) $\Rightarrow$ (1.) in Theorem \ref{thm:perfect}]
	Assume that $p$ satisfies $0<p<k+1$.  We will show that $I(n,k)$ is not perfect by showing it does not satisfy Decomposition \eqref{eq:InkPD}, or equivalently that  
	\begin{equation}
	\label{eq:ineq}
	\underbrace{\mathfrak{a}(n,k,k)+(x_1,\ldots,x_n)^{(k+1)}}_{I(n,k)}\neq \underbrace{\mathfrak{a}(n-1,k-1,k-1)+(x_1,\ldots,x_{n-1})^{(k)}}_{I(n-1,k-1)}\cap\left(\ydeal\right).
	\end{equation}
	We identify the space of square-free $x$-monomials with the monomial complete intersection $B=\F[x_1,\ldots,x_{n-1}]/(x_1^2,\ldots,x_{n-1}^2)$
	and $D=D_B=\frac{\partial}{\partial x_1}+\cdots+\frac{\partial}{\partial x_{n-1}}$ the associated lowering operator.  From our assumptions on $p$, Lemma \ref{lem:Vnkk} implies that the derivative map $D\colon B_k\rightarrow B_{k-1}$ is not surjective, and hence (by Lemma \ref{lem:Dsurj}), $V(n-1,k,k)=\ker(D)\cap B_k$.  In particular, there must exist a square-free polynomial of degree $k$, say $f=f(x_1,\ldots,x_{n-1})\in (x_1,\ldots,x_{n-1})^{(k)}$ with the property that $D(f)=0$ but $f\notin V(n-1,k,k)$.  Since $D(f)=0$, we deduce from Lemma \ref{lem:Dlem} that 
	$$f(x_1,\ldots,x_{n-1})\equiv f(y_1,\ldots,y_{n-1}) \equiv 0\ \ \text{mod} \ (y_1,\ldots,y_{n-1})^{(k)}+ (x_n^2).$$
	Therefore, $f\in I(n-1,k-1)\cap \left(\ydeal\right)$.  On the other hand, since $f\notin V(n-1,k,k)$ it follows from Lemma \ref{lem:restrict} that $f\notin V(n,k,k)$ either, and it follows that $f\notin I(n,k)$.  This shows that Inequality \eqref{eq:ineq} holds, and hence by Lemma \ref{lem:Nagata}, the Specht-monomial ideal $I(n,k)$ is not perfect.
\end{proof}

\subsection{Proof of Theorem \ref{thm:SpechtP} or Theorem \ref{thm:perfectSpecht}}
\label{sec:perfectSpecht}

\begin{proof}
	First assume that $p=0$ or $p\geq k+1$.  We prove by induction on $n\geq 3$ that for each integer $k$ satisfying $1\leq k<k+1\leq n-k$, and each $i$ satisfying $1\leq i\leq k$ the Specht ideal $\mathfrak{a}(n+1,i+1,i+1)$ is perfect.  First recall that  Corollary \ref{cor:CoC} says the change of coordinates map $\Phi$ gives a ring isomorphism
$$\mathfrak{a}(n+1,i+1,i+1)\cong \mathfrak{a}(n,i,i+1)=\mathfrak{a}(n,i,i)\cap (x_1,\ldots,x_n)^{(i+1)}$$
where the second equality follows from Theorem \ref{thm:B1} or Theorem \ref{thm:radD}.

For the base case $n=3$ and the only possible $k=1$ gives 
$$\mathfrak{a}(4,2,2)\cong \mathfrak{a}(3,1,2)=\mathfrak{a}(3,1,1)\cap (x_1,x_2,x_3)^{(2)}.$$
Note that $\mathfrak{a}(3,1,1)$ and $(x_1,x_2,x_3)^{(2)}$ are both perfect of grade $g=2$
Also note that the Specht-monomial ideal $I(3,1)=\mathfrak{a}(3,1,1)+(x_1,x_2,x_3)^{(2)}$ has grade $g+1=3$ and its quotient
$$\frac{\F[x_1,x_2,x_3]}{I(3,1)}\cong\frac{\F[z]}{(z^2)}$$
is Cohen-Macaulay.  Therefore $I(3,1)$ is perfect and hence by Lemma \ref{lem:HE}, so is $\mathfrak{a}(4,2,2)$.

For the inductive step, assume that $\mathfrak{a}(n,i,i)$ is perfect for every $1\leq i\leq k$.  We have 
$$\mathfrak{a}(n+1,i+1,i+1)\cong \mathfrak{a}(n,i,i)\cap (x_1,\ldots,x_n)^{(i+1)}.$$
Also note that since $1\leq i<i+1\leq k+1\leq n-k\leq n-i$, Theorem \ref{thm:perfect} implies that the Specht monomial ideal 
$$I(n,i)=\mathfrak{a}(n,i,i)+(x_1,\ldots,x_n)^{(i+1)}$$
is perfect of grade $g+1=n-i+1$.  Since $\mathfrak{a}(n,i,i)$ and $(x_1,\ldots,x_n)^{(i+1)}$ are both perfect of grade $g=n-i$.  Then it follows from Lemma \ref{lem:HE} that $\mathfrak{a}(n+1,i+1,i+1)\cong \mathfrak{a}(n,i,i)\cap (x_1,\ldots,x_n)^{(i+1)}$ is also perfect.  This proves implication (1.).

For (2.), assume that $n\geq 2p+1$.  By (1.), we see that $\mathfrak{a}(n,p,p)$ is perfect.  Also $(x_1,\ldots,x_n)^{(p)}$ is perfect by Lemma \ref{lem:PRSperfect}.  Therefore by Lemma \ref{lem:HE}, it follows that the Specht ideal $\mathfrak{a}(n+1,p+1,p+1)$ and the Specht-monomial ideal $I(n,p)$ are perfect or not, alike.  But Theorem \ref{thm:perfect} implies that the Specht-monomial ideal  
$$I(n,p)=\mathfrak{a}(n,p,p)+(x_1,\ldots,x_n)^{(p+1)}$$
is not perfect, and hence the Specht monomial ideal $\mathfrak{a}(n+1,p+1,p+1)$ is not perfect either.  
\end{proof}
As stated in the Introduction, we conjecture that Theorem \ref{thm:SpechtP}(2.) can be improved.  
\begin{conjecture}
	\label{conj:WM1}
	If $p=\operatorname{char}(\F)$ and $n,k$ are positive integers satisfying $0<p<k+1$ and $n\geq 2k+1$, then the Specht ideal $\mathfrak{a}(n+1,k+1,k+1)$ is not perfect. 
\end{conjecture}

	\begin{example}
		\label{ex:p23}
		Taking $p=2$, Theorem \ref{thm:SpechtP}(2.) says that $\mathfrak{a}(n+1,3,3)$ is not perfect for every $n\geq 5$, a result also obtained by Yanagawa \cite[Theorem 5.3]{Yana}.  For example if $n=5$, then $\mathfrak{a}(6,3,3)$ is not perfect, and one obstruction to perfection is the elementary symmetric polynomial $\alpha=e_2(x_1,\ldots,x_4)$ which lies in the intersection
		\begin{align*}
		e_2(x_1,\ldots,x_4)\in  & \mathfrak{a}(4,1,1)\cap \left((y_1,\ldots,y_4)^{(2)}+(x_5^2)\right)	= \bigcap_{\sigma\in \mathfrak{S}_5}\sigma.(x_1-x_2,x_1-x_3,x_1-x_4,x_1^2)
		\end{align*}
		but $e_2(x_1,\ldots,x_4)\notin I(5,2)$.  This indicates that the Specht-monomial ideal $I(5,2)$ must have an embedded prime divisor in characteristic $p=3$, which does not appear in higher characteristics.  Macaulay2 reveals that the primary decomposition of $I(5,2)$ over the field $\F=\Z/2\Z$ is 
		$$I(5,2)=\bigcap_{\sigma\in \mathfrak{S}_5}\sigma.(x_1-x_2,x_1-x_3,x_1-x_4,x_1^2)\cap Q$$
		where $Q$ is primary and satisfies $(x_1^2,\ldots,x_5^2)\subseteq Q$; in particular the maximal ideal $\mathfrak{m}=(x_1,\ldots,x_5)$ is an associated prime divisor of the Specht-monomial ideal $I(5,2)=\mathfrak{a}(5,2,2)+(x_1,\ldots,x_5)^{(3)}$ in characteristic $p=2$.  Since $\mathfrak{a}(5,2,2)$ and $(x_1,\ldots,x_5)^{(3)}$ are both perfect of grade $g=3$ in characteristic $p=2$ (and all characteristics), we deduce that $\mathfrak{a}(6,3,3)\cong \mathfrak{a}(5,2,3)=\mathfrak{a}(5,2,2)\cap (x_1,\ldots,x_5)^{(3)}$ is also not perfect by Lemma \ref{lem:HE}.  
		
		Similarly, for $p=3$, Theorem \ref{thm:SpechtP}(2.) implies that $\mathfrak{a}(n+1,4,4)$ is not perfect for $n\geq 7$ in characteristic $p=3$.  For example if $n=7$ then $\mathfrak{a}(8,4,4)$ is not perfect, and, as in the previous case, one can see that the Specht-monomial ideal $I(7,3)$ has an extra primary component that contains all the squared variables.  Conjecture \ref{conj:WM1} would imply that, for example, $\mathfrak{a}(n+1,5,5)$ is not perfect for all $n\geq 9$ in characteristic $p=3$.  Computations in Macaulay2 show that $\mathfrak{a}(10,5,5)$ is indeed not perfect, supporting this claim.  
	\end{example}

\subsection{Proof of Theorem \ref{thm:C}(3.) or Theorem \ref{thm:perfectD}}
The proof of Theorem \ref{thm:perfectD} is surprisingly similar to that of Theorem \ref{thm:radD}, and it too comes by way of several lemmas.
We want to show that for every $1\leq k<k+1\leq n-k$ the Specht-monomial ideal $I(n,k)=\mathfrak{a}(n,k,k)+(x_1,\ldots,x_n)^{(k+1)}$ satisfies
$$I(n,k)=I(n-1,k-1)\cap \left(\ydeal\right).$$
Of course there is an assumption about characteristic here, but we will not assume it yet, and try to point out exactly where we need it.  
\begin{lemma}
	\label{lem:easyC}
	Assume that $1\leq k<k+1\leq n-k$.  Then 
	$$I(n,k)\subseteq I(n-1,k-1)\cap \left(\ydeal\right).$$
\end{lemma}
\begin{proof}
	Certainly we have that $\mathfrak{a}(n,k,k)\subset\mathfrak{a}(n-1,k,k)\subset\mathfrak{a}(n-1,k-1,k-1)$, and also $(x_1,\ldots x_n)^{(k+1)}\subset (x_1,\ldots,x_{n-1})^{(k)}$, hence $I(n,k)\subset I(n-1,k-1)$ (if $k=1$, we should regard $I(n-1,0)$ as $R$).  It remains to see why $I(n,k)\subset\ydeal$.  For $T\in\stab(n,k,k)$ we have $$F^k_T(x_1,\ldots,x_n)=(x_{i_1}-x_{j_1})\cdots(x_{i_k}-x_n)=(y_{j_1}-y_{i_1})\cdots (y_{j_{k-1}}-y_{i_{k-1}})\cdot (-y_{i})\in \ydeal.$$
	For $S\in \stab(n,0,k+1)$, if $x_n\in\supp(M_S^{k+1})$ then $M_S^{k+1}=x_n\cdot M_{S'}^k$ for $S'\in\stab(n-1,0,k)$ and by Lemma \ref{lem:Dlem} we have 
	$$M_{S'}^k(x_1,\ldots,x_{n-1})=M_{S'}^k(y_1,\ldots,y_{n-1})+x_n\cdot D(M_{S'}^k(y_1,\ldots,y_{n-1})) \ \ \text{mod} \ (y_1,\ldots,y_{n-1})^{(k)}.$$
	It follows that $M_S^{k+1}=x_n\cdot M_{S'}^k\in \ydeal$.  If $x_n\notin \supp(M_S^{k+1})$, then it is obvious that $M_S^{k+1}\in\ydeal$.  Hence $I(n,k)\subseteq I(n-1,k-1)\cap \left(\ydeal\right)$, as desired. 
\end{proof}
We have added a superscript to our notation for the shifted Specht polynomials to help the reader remember degrees.
The other containment 
\begin{equation}
\label{eq:InkC}
I(n,k)\supseteq I(n-1,k-1)\cap \left(\ydeal\right)
\end{equation}
is harder to prove.

Note that the ideal $I(n-1,k-1)$ is generated in degrees $k-1$ and $k$ by the following subspaces of forms:
\begin{align*}
V= & V(n-1,k-1,k-1)= & \left\langle F_T^{k-1} \ | \ T\in\stab(n-1,k-1,k-1)\right\rangle\\
U= & U(n-1,k-1,k)= & \left\langle x_nF_T^{k-1} \ | \ T\in\stab(n-1,k-1,k-1)\right\rangle+\left\langle M^k_S \ | \ S\in\stab(n-1,0,k)\right\rangle\\
 =& & x_n\cdot V(n-1,k-1,k-1)+V(n-1,0,k) 
\end{align*}
We make some preliminary observations, but first some notation.  Denote by $\N^n(m)$ the set of exponent vectors of degree, i.e. $\a=(a_1,\ldots,a_n)$ with $a_1+\cdots+a_n=m$.  A monomial in the $y$-variables (resp. the $x$-variables) will be denoted by $\y^\a=y_1^{a_1}\cdots y_n^{a_n}$ (resp. $\x^\a=x_1^{a_1}\cdots x_n^{a_n}$), and its radical is the square-free monomial $\sqrt{\y^\a}=\prod_{a_i>0}y_i$.  We also define the \emph{weight} of a monomial to be $\operatorname{wt}(\y^\a)=\#\{a_i>0\}$, the number of non-zero entries in its exponent vector.
\begin{lemma}
	\label{lem:VUP}
	With $U$ and $V$ as above, we have
	\begin{enumerate}
		\item $x_n\cdot V\subseteq U$, 
		\item $x_n\cdot U\subseteq I(n,k)$, and 
		\item for every $P\in I(n-1,k-1)$, and for all exponent vectors $\a\in\N^{n-1}(m)$ and $\b\in\N^{n-1}(m-1)$, there exists elements $\nu_\a\in V$ and $\mu_\b\in U$ such that 
		\begin{align}
		\label{eq:P}
		P\equiv & \sum_{\a\in\N^{n-1}(m)}\y^\a\nu_\a +\sum_{\b\in\N^{n-1}(m-1)}\y^\b\mu_\b & \text{mod} \ I(n,k).
		\end{align} 
	\end{enumerate}
\end{lemma}
\begin{proof}
	(1.) is obvious from the definitions.  For (2.), note that for $S\in\stab(n-1,0,k-1)$, $x_n\cdot M_S^k\in V(n,0,k+1)\subset I(n,k)$.  Also for $T\in\stab(n-1,k-1,k-1)$ and for index $1\leq i\leq n-1$ such that $i\notin \supp(T)$, which exists because $n-1\geq 2k>2(k-1)$, we have 
	$$x_n^2\cdot F_T^{k-1}=x_n(x_n-x_i)\cdot F_T^{k-1}+x_nx_i\cdot F_T^{k-1}$$
	and since $(x_n-x_i)\cdot F_T^{k-1}\in V(n,k,k)$ and $x_nx_i\cdot F_T^{k-1}\in V(n,0,k+1)$, it follows that $x_n^2\cdot F_T^{k-1}\in I(n,k)$, and (2.) follows.  Finally fix a homogeneous polynomial $P\in I(n-1,k-1)$.  Then for each $T\in\stab(n-1,k-1,k-1)$ and each $S\in\stab(n-1,0,k)$, there exist polynomials, which we may take in the $y$-variables, $g_T(\y)$ of degree $m$ and $h_S(\y)$ of degree $m-1$ for which 
	$$P=\sum_{T \in\stab(n-1,k-1,k-1)}g_T(\y)\cdot F_T^{k-1}+\sum_{S\in \stab(n-1,0,k)}h_S(\y)\cdot M_S.$$
	Writing $g_T(\y)=g^0_T(y_1,\ldots,y_{n-1})+x_ng^1_T(y_1,\ldots,y_{n-1})+x_n^2g^2_T(\y)$ and also $h_S(\y)=h^0_S(y_1,\ldots,y_{n-1})+x_nh^1_S(\y)$, it follows from (2.) that $x_n^2g^2_T(\y)\cdot F_T^{k-1}\in I(n,k)$, $x_nh^1_S(\y)\cdot M_S^k\in I(n,k)$, and also that $x_ng^1_T(y_1,\ldots,y_{n-1})\cdot F_T^{k-1}\in U$.  Then taking monomial expansions, reversing orders of summations, and grouping like monomial terms, we get 
	\begin{align*}
	P\equiv  & \sum_{T \in\stab(n-1,k-1,k-1)}\sum_{\a\in\N^{n-1}(m)}c_T^0(\a)\y^\a\cdot F_T^{k-1}\\
	& +\sum_{T \in \stab(n-1,k-1,k-1)}\sum_{\b\in\N^n(m-1)}c_T^1(\b)\y^\b\cdot x_nF_T^{k-1}+\sum_{S\in \stab(n-1,0,k)}\sum_{\b\in\N^n(m-1)}d_S^0(\b)\y^\b\cdot M_S\\
	= & \sum_{\a\in\N^n(m)}\y^\a\cdot\left(\sum_{T \in \stab(n-1,k-1,k-1)}c_T^0(\a)F_T^{k-1}\right)\\
	&  + \sum_{\b\in\N^{n-1}(m-1)}\y^\b\cdot \left(\sum_{T \in \stab(n-1,k-1,k-1)}c_T^1(\b)F_T^{k-1}+\sum_{S\in \stab(n-1,0,k)}d_S^0(\b)M_S^k\right) & \text{mod} \ I(n,k)
	\end{align*}
	and (3.) follows.
\end{proof}

The following Lemma is analogous to Lemma \ref{lem:monomialrad} in Section \ref{sec:radical}.
\begin{lemma}
	\label{lem:monomialD}
	If $P\in I(n-1,k-1)\cap \left(\ydeal\right)$ is expressed as in \eqref{eq:P}, then each of its monomial summands must lie in $\ydeal$ is too, i.e. $\y^\a\cdot \mu_\a\in \ydeal$ and $\y^\b\cdot\nu_\b\in\ydeal$ for all $\a\in\N^{n-1}(m)$ and $\b\in\N^{n-1}(m-1)$.  
\end{lemma}
\begin{proof}
	By Lemma \ref{lem:VUP}, we may write 
	$$P=\underbrace{\sum_{\a\in\N^{n-1}(m)}\y^\a\nu_\a}_{P_1} +\underbrace{\sum_{\b\in\N^{n-1}(m-1)}\y^\b\mu_\b}_{P_2}+P_3$$
	for some $\nu_\a\in V$, some $\mu_\b\in U$, and some $P_3\in I(n,k)$.  Next note that if $P\in\ydeal$ then both $P_1\in\ydeal$ and $P_2\in\ydeal$.  Indeed note that in their respective $y$-monomial expansions, those monomials in $P_1$ are all independent of $y_n$, whereas all monomials in the $y$-monomial expansion of $P_2$ are either divisible by $y_n=x_n$, or in $\ydeal$ already, by Lemma \ref{lem:Dlem}.  
	
	For the monomial products in $P_1$ we assume by way of contradiction that $\y^\a\cdot\nu_\a\notin\ydeal$ for some $\a\in\N^{n-1}$.  Then since $\ydeal$ is a monomial ideal, it follows that in the $y$-monomial expansion of $\y^\a\nu_\a$, there must be some monomial say $\y^\d$ which is not in $\ydeal$.  Since $\y^\a\nu_\a$ is independent of $x_n$, so is $\y^\d$.  Since $(y_1,\ldots,y_{n-1})^{(k)}$ consists of all $y$-monomials of weight at least $k$, it follows that $\operatorname{wt}(\y^\d)\leq k-1$.  On the other hand, every $y$-monomial in the monomial expansion of $\nu_\a$ has weight equal to $k-1$, hence it follows that $\operatorname{wt}(\y^\d)=k-1$.  But then we can deduce, as in the proof of Lemma \ref{lem:monomialrad}, that 
	$$\frac{\y^\d}{\sqrt{\y^\d}}=\y^\a$$
	and hence the exponent vector $\d$ uniquely determines the exponent vector $\a$.  This implies that $\y^\d$ occurs with the same coefficient in the monomial expansion of the term $\y^\a\nu_\a$ as it does in the entire sum
	$$P_1=\sum_{\a\in\N^{n-1}(m)}\y^\a\nu_\a,$$
	contradicting the fact that $P_1\in\ydeal$.
	The argument for $P_2$ is similar.	
\end{proof}

Lemma \ref{lem:monomialD} says that to check the containment 
$$I(n,k)\supseteq I(n-1,k-1)\cap \left(\ydeal\right)$$
it suffices to check on products of monomials and forms in either $V$ or $U$, i.e.  for any $\a,\b\in\N^{n-1}$ and for any $\nu\in V$ and any $\mu\in U$
\begin{align}
\label{eq:impV}
\y^\a\nu\in \ydeal & \Rightarrow & \y^\a\nu\in I(n,k)\\
\label{eq:impU}
\y^\b\mu\in \ydeal & \Rightarrow & \y^\b\mu\in I(n,k) 
\end{align}
The next lemma verifies implications \eqref{eq:impV} and \eqref{eq:impU} in the special case where $\a,\b=\mathbf{0}$.  This seems to be where we need our assumptions on $p$.
\begin{lemma}
	\label{lem:zero}
	Fix $\nu\in V(n-1,k-1,k-1)$ and $\mu\in U(n-1,k-1,k)$.  If $\nu\in \ydeal$ then $\nu=0$.  If $p=0$ or $p\geq k+1$ and if $\mu\in \ydeal$ then $\mu\in \mathfrak{a}(n,k,k)\subset I(n,k)$.
\end{lemma}
\begin{proof}
	The first statement for $\nu$ is obvious for degree reasons.  For the second statement, we can write $\mu=\mu(x_1,\ldots,x_n)=x_{n}\alpha+\beta$ for polynomials $\alpha=\alpha(x_1,\ldots,x_{n-1})\in V_{\x}(n-1,k-1,k-1)$ and $\beta=\beta(x_1,\ldots,x_{n-1})\in V_{\x}(n-1,0,k)$.  Here it will be important to distinguish between polynomials in $x$-variables and those in the $y$-variables, hence we adopt the notation $V_{\x}(m,j,j)$ and $V_{\y}(m,j,j)$ to denote the $\F$-span of Specht polynomials in the $x$-variables and $y$-variables, respectively; note that if $m\leq n-1$ these two subspaces coincide.  
	
	Consider inclusions of monomial complete intersections, $B\subset A$ defined in the $y$-variables by:
	$$B=\frac{\F[y_1,\ldots,y_{n-1}]}{(y_1^2,\ldots,y_{n-1}^2)}\hookrightarrow A= \frac{\F[y_1,\ldots,y_{n}]}{(y_1^2,\ldots,y_n^2)}$$
	and let $L_B,D_B,H_B$, and $L_A,D_A,H_A$  be their respective raising, lowering, and semi-simple operators, respectively, as in Section \ref{sec:Lefschetz}.  Then since $p=0$ or $p\geq k+1$, Lemma \ref{lem:Vnkk} implies that the lowering maps for $B$ and $A$,
	$$D_B\colon B_k\rightarrow B_{k-1}, \ \ \text{and} \ \ D_A\colon A_k\rightarrow A_{k-1}$$
	are both surjective, which by Lemma \ref{lem:Dsurj} is equivalent to their primitive subspaces $P_{B,k}=\ker(D_B)\cap B_k$ and $P_{A,k}=\ker(D_A)\cap A_k$ satisfying 
	$$P_{B,k}=V_{\y}(n-1,k,k), \ \ \text{and} \ \ P_{A,k}=V_{\y}(n,k,k).$$
	We will identify $B$ (resp. $A$) with the subspace spanned by square-free monomials in variables $y_1,\ldots,y_{n-1}$ (resp. $y_1,\ldots,y_n$).
	%Recall the formula:  For each square-free polynomial $f(x_1,\ldots,x_{n-1})\in R$ we have 
	%\begin{equation}
	%\label{eq:equiv}
	%f(x_1,\ldots,x_{n-1})\equiv (-1)^k\left(f(y_1,\ldots,y_{n-1})-x_nD_B\left(f(y_1,\ldots,y_{n-1})\right)\right).
	%\end{equation}
	
	Then if $\mu=x_n\alpha(x_1,\ldots,x_{n-1})+\beta(x_1,\ldots,x_{n-1})\in \ydeal$, by Lemma \ref{lem:Dlem} we have 
	\begin{align}
	\label{eq:ab}
	\mu\equiv &  (-1)^{k-1}x_n\left(\alpha(y_1,\ldots,y_{n-1})+D_B(\beta(y_1,\ldots,y_{n-1}))\right)\equiv 0 & \text{mod} \ \ydeal
	\end{align}
	(note that $D_B(\alpha(y_1,\ldots,y_{n-1}))=0$ and $\beta(y_1,\ldots,y_{n-1})\in \ydeal$ automatically).  Dividing by $x_n$ in \eqref{eq:ab} we see that 
	$$\alpha(y_1,\ldots,y_{n-1})+D_B(\beta(y_1,\ldots,y_{n-1}))\in \left((y_1,\ldots,y_{n-1})^{(k)}+x_n^2\right):x_n=(y_1,\ldots,y_{n-1})^{(k)}+(x_n)$$
	which, for degree reasons, implies that 
	$$\alpha(y_1,\ldots,y_{n-1})+D_B(\beta(y_1,\ldots,y_{n-1}))=0.$$
	Therefore $\beta_{\y}=\beta(y_1,\ldots,y_{n-1})\in B_k$ is a square-free polynomial in $y_1,\ldots,y_{n-1}$ such that $D_B(\beta_{\y})=-\alpha_{\y}\in V(n-1,k-1,k-1)=P_{B,k-1}$.  Applying the commutator relation 
	$$\left[D_B,L_B\right]=H_B$$
	to $\alpha_\y=\alpha(y_1,\ldots,y_{n-1})$, we find that 
	$$D_B\circ L_B\left(\alpha_{\y}\right)=H\left(\alpha_{\y}\right)=\left(n-1-2(k-1)\right)\cdot \alpha_{\y}$$
	which implies that 
	\begin{equation}
	\label{eq:Lb}
	\mu_B(y_1,\ldots,y_{n-1})\coloneqq L_B\left(\alpha_{\y}\right)+(n-2k+1)\cdot \beta_{\y}\in P_{B,k}=V_{\y}(n-1,k,k).
	\end{equation}
	Note that Equation \eqref{eq:Lb} also implies that $\mu_B\coloneqq \mu_B(x_1,\ldots,x_{n-1})\in V_{\x}(n-1,k,k)$.  We can also apply the commutator relations for those operators on $A$.  Note that the restriction of $D_A=D_B+\partial/\partial y_n$ to $B$ is $D_B$, and that by Lemma \ref{lem:restrict}, we have 
	$$V_{\y}(n-1,k,k)=V_{\y}(n,k,k)\cap B_k.$$
	It follows that $D_A(\beta_{\y})=D_B(\beta_{\y})=-\alpha_{\y}\in V_{\y}(n-1,k-1,k-1)\subset V_{\y}(n,k-1,k-1)$, and hence applying the commutator relation 
	$$\left[D_A,L_A\right]=H_A$$
	to $D_A(\beta_{\y})=-\alpha_{\y}$ we also find that
	\begin{equation}
	\label{eq:La}
	\mu_A(y_1,\ldots,y_{n-1},y_n)\coloneqq L_A(\alpha_{\y})+(n-2k+2)\cdot\beta_{\y}\in P_{A,k}=V_{\y}(n,k,k)
	\end{equation}
	and hence $\mu_A=\mu_A(x_1,\ldots,x_n)\in V_{\x}(n,k,k)$.
	Noting that $L_A=L_B+x_n$ we see that adding $\mu$ to $\mu_B$ yields $\mu_A$, i.e. 
	\begin{align*}
	\underbrace{\left(x_n\alpha+\beta\right)}_{\mu}+\underbrace{\left((x_1+\cdots+x_{n-1})\cdot \alpha+(n-2k+1)\beta\right)}_{\mu_B}=\underbrace{(x_1+\cdots+x_n)\cdot\alpha+(n-2k+2)\beta}_{\mu_A}.
	\end{align*}
	It follows from \eqref{eq:Lb} and \eqref{eq:La} that $\mu=x_n\cdot\alpha+\beta=\mu_A-\mu_B\in V_{\x}(n,k,k)=V(n,k,k)\subset\mathfrak{a}(n,k,k)$, as claimed.
\end{proof}

Next we check containment \eqref{eq:InkC} for monomials which are not contained in the support of $\nu$ and $\mu$.  The following Lemma is analogous to Lemma \ref{lem:suppNo} in Section \ref{sec:radical}.  Recall that the support of the (shifted) Specht polynomial for $T$ is the set of square-free monomials indexed by subsets of numbers in the support of $T$, no two of which lie in the same column of $T$.  
\begin{lemma}
	\label{lem:supportNo}
	Fix tableaux $T\in\stab(n-1,k-1,k-1)$ and $S\in\stab(n-1,0,k)$, and exponent vectors $\a,\b\in\N^{n-1}$.  
	\begin{enumerate}
		\item If $\sqrt{\x^\a}\not\in \operatorname{supp}(F_T^{k-1})$ then $\y^\a\cdot F_T^{k-1}\in \mathfrak{a}(n,k,k)\subset I(n,k)$.
		\item If $\sqrt{\x^\b}\notin \operatorname{supp}(M_S^k)$ then $\y^\b\cdot M_S^k\in (x_1,\ldots,x_n)^{(k+1)}\subset I(n,k)$.
	\end{enumerate}
\end{lemma} 
\begin{proof}
	For (1.) assume $\sqrt{\x^\a}\notin\supp(F_T^{k-1})$.  If some variable $x_i\in\supp(\y^\a)$ but $x_i\notin\supp(F_T^{k-1})$, then clearly 
	$$\y^\a\cdot F_T^{k-1}=\y^{\a'}\cdot y_i\cdot F_T^{k-1}=\y^{\a'}\cdot (x_n-x_i)\cdot F_{T}^{k-1}\in V(n,k,k).$$
	We may therefore assume that every variable of $\x^\a$ lies in the support of $F_T^{k-1}$.  But since $\sqrt{\x^\a}\notin\supp(F_T^{k-1})$ there must be two distinct indices $i,j$ for which $y_i,y_j\in\supp(\y^\a)$ and $i,j$ lie in the same column in $T$.  Choose any index $r\neq i,j$ such that $x_r\notin\supp(F_T^{k-1})$--presumably there is such an $r$ since we are assuming that $k+1\leq n-k$--and let $(i,r), (j,r)\in\mathfrak{S}_n$ be the transpositions swapping $i,r$ and $j,r$, respectively.  Then we have 
	$$\y^\a\cdot F_T^{k-1}=\y^\a\cdot\left(F_{(i,r).T}^{k-1}+F_{(j,r).T}^{k-1}\right)$$
	and since $x_i\notin\supp(F_{(i,r).T}^{k-1})$, we must have 
	$$\y^\a\cdot F_{(i,r).T}^{k-1}=\y^{\a'}\cdot y_i\cdot F_{(i,r).T}^{k-1}=\y^{\a'}\cdot (x_n-x_i)\cdot F_{(i,r).T}^{k-1}\in \mathfrak{a}(n,k,k)\subset I(n,k)$$
	and similarly for $F^{k-1}_{(j,r).T}$.  This proves that $\y^\a\cdot F_T^{k-1}\in I(n,k)$.  Item (2.) is easier and left to the reader.
\end{proof}
	
	Finally, we need to check what containment \eqref{eq:InkC} for monomials which are contained in the support of $\nu$ or $\mu$.  Since the ideal $\ydeal$ is generated by square-free monomials in $(y_1,\ldots,y_{n-1})$ and the monomial $x_n^2$, it suffices to prove implications \eqref{eq:impV} and \eqref{eq:impU} for square free $\y^\a$.  Also, by symmetry, it will suffice to assume that $\y^\a=\y^\m=y_1\cdots y_m$.  As in Corollary \ref{cor:inductive}, we define the subset of standard tableau on a shape $\lambda$ as those which contain the set of integers $\{1,\ldots,m\}$ in their support, denoted by $\stab_m(\lambda)\subset \stab(\lambda)$, and define the subspaces
	\begin{align*}
	V_m= & V_m(n-1,k-1,k-1)= &\left\langle F_T^{k-1} \ | \ T\in\stab_m(n-1,k-1,k-1)\right\rangle\\
	U_m= & U_m(n-1,k-1,k)= &\left\langle x_nF_T^{k-1} \ | \ T\in\stab_m(n-1,k-1,k-1)\right\rangle+\left\langle M^k_S \ | \ S\in\stab_m(n-1,0,k)\right\rangle 
	\end{align*}
	We also define their complimentary subspaces
	\begin{align*}
	V^m= & V^m(n-1,k-1,k-1)= &\left\langle F_T^{k-1} \ | \ T\in\stab(n-1,k-1,k-1)\setminus \stab_m(n-1,k-1,k-1)\right\rangle\\
	U^m= & U^m(n-1,k-1,k)= & \left\langle x_nF_T^{k-1} \ | \ T\in\stab(n-1,k-1,k-1)\setminus \stab_m(n-1,k-1,k-1)\right\rangle\\
	 & & +\left\langle M^k_S \ | \ S\in\stab(n-1,0,k)\setminus \stab_m(n-1,0,k)\right\rangle. 
	\end{align*}
	Then we have vector space decompositions $V= V_m\oplus V^m$ and $U= U_m+ U^m$, and Lemma \ref{lem:supportNo} implies that $\y^\m\cdot\alpha\in I(n,k)$ for $\alpha\in V^m\sqcup U^m$.  Hence it only remains to check $\alpha\in V_m\sqcup U_m$.  
	As in Corollary \ref{cor:inductive} we have the following bijective map of vector spaces:
	$$\xymatrixrowsep{.5pc}\xymatrix{V_m(n-1,k-1,k-1)\ar[r] & V([n-1]_m,k-1-m,k-1-m)\\
	F_T^{k-1}\ar@{|->}[r] & F_{T'}^{k-1-m}\\}$$
	where 
	\begin{equation}
	\label{eq:T}
	T=\begin{ytableau}
	\none & \none & \none & 1 & \cdots & m & i_{m+1} & \cdots & i_{k}\\
	i_{k+1} & \cdots & i_{n-k} & j_1 & \cdots & j_m & j_{m+1} & \cdots & j_k\\
	\end{ytableau}
	\end{equation}	
	and 
	\begin{equation}
	\label{eq:T'}
	T'=\begin{ytableau}
	\none & \none & \none & \none & \none & \none & i_{m+1} & \cdots & i_{k}\\
	i_{k+1} & \cdots & i_{n-k} & j_1 & \cdots & j_m & j_{m+1} & \cdots & j_k\\
	\end{ytableau}
	\end{equation}
	We also have the (possibly non-injective) linear map 
	$$\xymatrixrowsep{.5pc}\xymatrix{U_m(n-1,k-1,k-1)\ar[r] & U([n-1]_m,k-1-m,k-m)\\
		x_nF_T^{k-1}\ar@{|->}[r] & x_nF_{T'}^{k-1-m}\\
		M_S^k\ar@{|->}[r] & M_{S'}^{k-m}\\}$$
	where $T\in\stab_m(n-1,k-1,k-1)$ and $T'\in\stab([n-1]_m,k-1-m,k-1-m)$ are the tableau in \eqref{eq:T} and \eqref{eq:T'}, respectively, and where 
	where $S\in\stab(n-1,0,k)$ and $S'\stab([n-1]_m,0,k-m)$ are given by
	\begin{equation}
	\label{eq:S}
	S=\begin{ytableau}
	\none & \none & \none & 1 & \cdots & m & i_{m+1} & \cdots & i_{k}\\
	i_{k+1} & \cdots & i_{n-k} \\
	\end{ytableau}
	\end{equation}	
	and 
	\begin{equation}
	\label{eq:S'}
	S'=\begin{ytableau}
	\none & \none & \none &  i_{m+1} & \cdots & i_{k}\\
	i_{k+1} & \cdots & i_{n-k}\\
	\end{ytableau}
	\end{equation}
	
	The following two lemmas, will be useful.
	\begin{lemma}
		\label{lem:squared}
		For any $\beta\in V([n-m],k-m,k-m)$ where $[n-m]=\{m+1,\ldots,n\}$, we have
		$$\left(\x^\m\right)^2\cdot\beta\in I(n,k).$$ 
	\end{lemma}	
	\begin{proof}
		Fix a tableau ${T'}\in\stab([n-m],k-m,k-m)$ as in \eqref{eq:T'} and let $T\in \stab_m(n-1,k-1,k-1)$ be the corresponding tableau as in \eqref{eq:T}.  Then it suffices to show that we have 
		$$\left(\x^\m\right)^2\cdot F_{T'}^{k-m}=x_1^2\cdots x_m^2\cdot F_{T'}^{k-m}\in I(n,k)=\mathfrak{a}(n,k,k)+(x_1,\ldots,x_n)^{(k+1)}.$$
		On the other hand we clearly have 
		$$x_1\cdots x_m\cdot F_T^k=x_1^2\cdots x_m^2\cdot F_{T'}^{k-m}+\left(\text{stuff in} \ (x_1,\ldots,x_n)^{(k+1)}\right)\in \mathfrak{a}(n,k,k)$$
		and the result follows.
	\end{proof}
		
	\begin{lemma}
		\label{lem:mdiff1}
		Fix elements $\nu\in V_m(n-1,k-1,k-1)$ and $\mu\in U_m(n-1,k-1,k)$, and let $\nu'\in V([n-1]_m,k-1-m,k-1-m)$ and $\mu'\in U([n-1]_m,k-1-m,k-m)$ their respective images under the maps above.  Then we have 
		\begin{enumerate}
			\item $\y^\m\left(\nu-(-1)^m\y^\m\nu'\right)\in \ydeal$,
			\item $\y^\m\left(\mu-(-1)^m\y^\m\mu'\right)\in \ydeal$,
			\item $\x^\m\left(\mu-\x^\m\mu'\right)\in I(n,k)$, and
			\item $\y^\m\cdot \mu-(-1)^m\x^\m\cdot \mu\in I(n,k)$.
		\end{enumerate}
	\end{lemma}
\begin{proof}
	For (1.) note that it suffices to see that for $T\in\stab_m(n-1,k-1,k-1)$ and $T'\in\stab([n-1]_m,k-1-m,k-1-m)$ as in \eqref{eq:T} and \eqref{eq:T'} we have
	$$\y^\m\left(F_T^{k-1}-(-1)^m\y^\m\cdot F_{T'}^{k-1}\right)\in \ydeal.$$ 
	By Lemma \ref{lem:Dlem}, we have $F^{k-1}_{T,\y}=F_T^{k-1}(y_1,\ldots,y_{n-1})=(-1)^{k-1}\cdot F_T^{k-1}(x_1,\ldots,x_{n-1})=F_T^{k-1}$ and $F_{T',\y}^{k-1-m}=(-1)^{k-1-m}\cdot F_{T'}^{k-1}$.  Since the difference $F_{T,\y}^{k-1}-\y^\m\cdot F_{T',\y}^{k-1}$ is a combination of $y$-monomials which do not contain $\y^\m$ in their support, it follows that the product
	$$\y^\m\left(F_T^{k-1}-(-1)^m\y^m\cdot F_{T'}^{k-1}\right)\in\ydeal.$$

	For (2.), write $\mu=x_n\cdot \alpha+\beta$ where $\alpha\in V_m(n-1,k-1,k-1)$ and $\beta\in V_m(n-1,0,k)$.  By (1.) we know that 
	$$\y^\m\left(x_n\cdot\alpha-(-1)^m x_n\cdot\alpha'\right)\in\ydeal$$
	hence it suffices to take $\mu=\beta=M_S^k\in V_m(n-1,0,k)$ and show that 
	$$\y^\m\left(M_S^k-(-1)^m\y^\m M_{S'}^{k-m}\right)\in\ydeal$$
	for $S\in\stab_m(n-1,0,k)$ and $S'\in\stab([n-1]_m,0,k-m)$ as in \eqref{eq:S} and \eqref{eq:S'}.  In this case, we have 
	$$M_S^k=x_1\cdots x_m\cdot M_{S'}^{k-m}$$
	and Lemma \ref{lem:Dlem} implies that the difference satisfies 
	\begin{align*}
	M_S^k-(-1)^m\y^\m\cdot M_{S'}^{k-m}= & (-1)^kM_{S,\y}^k+(-1)^{k-1}x_nD(M^k_{S,\y})\\
	& -(-1)^m\y^\m\left((-1)^{k-m}M_{S',\y}^{k-m}+(-1)^{k-1-m}x_nD(M_{S',\y}^{k-m})\right) & \text{mod} \ (x_n^2)\\
	= & (-1)^{k-1}x_n\cdot\left(D(M_{S,\y}^k)-\y^\m\cdot D(M_{S'}^{k-m})\right) & \text{mod} \ (x_n^2)\\
	= & (-1)^{k-1}x_n\cdot \left(\sum_{i=1}^my_1\cdots \hat{y_i}\cdots y_m\cdot M_{S',\y}^{k-m}\right)	& \text{mod} \ (x_n^2)
	\end{align*}
	Therefore it follows that the product satisfies
	$$\y^m\left(M_S^k-(-1)^m\y^\m\cdot M_{S'}^{k-m}\right)\in\ydeal.$$
		
	For (3.), we write $\mu=x_n\alpha+\beta$ where $\alpha\in V_m(n-1,k-1,k-1)$ and $\beta\in V_m(n-1,0,k)$, and also $\mu'=x_n\alpha'+\beta'$.  Then we have 
	$$\x^\m\left(x_n\alpha-\x^\m x_n\alpha'\right)\in (x_1,\ldots,x_n)^{(k+1)}\subset I(n,k)$$
	and also 
	$$\x^m\left(\beta-\x^\m\beta'\right)=0\in I(n,k)$$
	and the result follows.
	
	For (4.), note that 
	\begin{align*}
	\y^\m-(-1)^m\x^\m= & (x_n-x_1)\cdots(x_n-x_m)-(-1)^mx_1\cdots x_m=x_n\left(\text{stuff}\right)
	\end{align*}
	and since $x_n\cdot \mu\in I(n,k)$ by Lemma \ref{lem:VUP}, the result follows.
\end{proof}

The following is an analogue of Lemma \ref{lem:suppYes} in Section \ref{sec:radical}.
\begin{lemma}
	\label{lem:supportYes}
	If $\y^\m\cdot\nu\in \ydeal$ for some $\nu\in V_m$, then $\nu=0$, and in particular, $\y^\m\cdot \nu\in I(n,k)$.  If $p=0$ or $p\geq k+1$, then if $\y^\m\cdot\mu\in\ydeal$ for some $\mu\in U_m$, then $\y^\m\cdot\mu\in I(n,k)$.
\end{lemma}
\begin{proof}
	First assume that for some $\nu\in V_m$ we have $\y^\m\cdot \nu\in\ydeal$.  By Lemma \ref{lem:mdiff1}(1.) we have 
	$$\y^\m\left(\nu-\y^\m\cdot\nu'\right)\in \ydeal$$
	and it follows also that 
	$$\left(\y^\m\right)^2\cdot\nu'\in (y_1,\ldots,y_{n-1})^{(k)}+(x_n^2).$$
	Therefore we have  
	$$\nu'\in \left(\ydeal:\left(\y^\m\right)^2\right)=\left(\ydeal:\y^\m\right)=(y_{m+1},\ldots,y_{n-1})^{k-m}+(x_n^2).$$
	and hence that $\nu'\in V([n-1]_m,k-1-m,k-1-m)\cap \left((y_{m+1},\ldots,y_{n-1}\right)^{(k-m)}+(x_n^2)$.  By Lemma \ref{lem:zero}, it follows that $\nu'=0$, and therefore also $\nu=0$, which proves the first statement.
	
	Next, assume that $p=0$ or $p\geq k+1$, and assume that for some $\mu=\mu(x_1,\ldots,x_{n-1})\in U_m$ we have $\y^\m\cdot\mu\in \ydeal$.  Then by Lemma \ref{lem:mdiff1}(2.), we must also have $\left(\y^\m\right)^2\mu'\in\ydeal$, and hence also we must have 
	$$\mu'\in\left(\ydeal:\left(\y^\m\right)^2\right)=\left(\ydeal:\y^\m\right)=(y_{m+1},\ldots,y_{n-1})^{k-m}+(x_n^2).$$
	Therefore we have $\mu'\in U([n-1]_m,k-1-m,k-m)\cap (y_{m+1},\ldots,y_{n-1})^{k-m}+(x_n^2)$, and it follows from Lemma \ref{lem:zero} that $\mu'\in V([n]_m,k,k)$.  Therefore by Lemma \ref{lem:squared}, it follows that $\left(\x^\m\right)^2\cdot\mu'\in I(n,k)$.  Then by Lemma \ref{lem:mdiff1}(3.) we must also have $\x^m\cdot \mu\in I(n,k)$, and from Lemma \ref{lem:mdiff1}(4.) it follows that $\y^\m\cdot \mu\in I(n,k)$, as desired.
\end{proof}	

Finally we are in a position to prove Theorem \ref{thm:perfectD}:
\begin{proof}[Proof of Theorem \ref{thm:perfectD}]
	Assume that $p=0$ or $p\geq k+1$.  By Lemma \ref{lem:easyC}, we have 
	$$I(n,k)\subseteq I(n-1,k-1)\cap\left(\ydeal\right).$$
	For the reverse containment, Lemma \ref{lem:monomialD} implies we only have to check on products of monomials in $(y_1,\ldots,y_{n-1})$ and forms in the subspaces $V=V(n-1,k-1,k-1)$ and $U=U(n-1,k-1,k)$.  Since $\ydeal$ is generated by square-free monomials in $(y_1,\ldots,y_{n-1})$ it follows that we may assume our monomials are square-free, and by symmetry we may assume that our monomial is $\y^\m=y_1\cdots y_m$.  Write $V=V_m\oplus V^m$ and $U=U_m+ U^m$ as above.  Then Lemma \ref{lem:supportNo} implies that $\y^\m\cdot \alpha\in I(n,k)$ for any $\alpha\in V^m\sqcup U^m$.  Also if $\beta\in V_m\sqcup U_m$, then Lemma \ref{lem:supportYes} implies that if $\y^\m\cdot\beta\in \ydeal$, then $\y^\m\cdot\beta \in I(n,k)$, and the result follows. 
	
	Conversely, assume that $0<p<k+1$.  Then as in the proof of Theorem \ref{thm:perfect} and in particular \eqref{eq:ineq}, we have $I(n,k)\neq I(n-1,k-1)\cap \left(\ydeal\right)$.
\end{proof}

\end{document}